\documentclass[12pt]{amsart}

\usepackage{amssymb,stmaryrd}
\usepackage{booktabs}
\usepackage{cases}
\usepackage{array} 
\usepackage{float} 
\usepackage{rotating}
\usepackage{youngtab}
\usepackage{ytableau}
\usepackage{tikz-cd}
\usetikzlibrary{decorations.fractals}
\usetikzlibrary{cd}
\usetikzlibrary{arrows}
\usepackage{youngtab}
\usepackage{ytableau}
\usepackage{caption}
\usepackage{subcaption}
\usepackage{lipsum}
  \usepackage[titletoc]{appendix}
\usepackage{graphicx}
\usepackage{pgf,tikz}
\usepackage{mathrsfs}
\usetikzlibrary{arrows}
\usepackage[export]{adjustbox}
\usepackage[foot]{amsaddr}
\makeatletter
\renewcommand{\email}[2][]{%
  \ifx\emails\@empty\relax\else{\g@addto@macro\emails{,\space}}\fi%
  \@ifnotempty{#1}{\g@addto@macro\emails{\textrm{(#1)}\space}}%
  \g@addto@macro\emails{#2}%
}
\makeatother
\usepackage{float}
\usepackage{mathtools}
\usepackage{enumerate,pdflscape}
\usepackage{multirow}
\usepackage{color,colorbl}
\definecolor{Gray}{gray}{0.9}
\definecolor{LightCyan}{rgb}{0.88,1,1}

\usepackage[bindingoffset=0.2in,%
            left=1in,right=1in,top=1in,bottom=1in,%
            footskip=.25in]{geometry}

 \newcommand{\ad}{\text{ad}}

 \newcommand{\h}{\mathfrak{h}}

  \theoremstyle{definition}
  \newtheorem{definition}{Definition}[section]
  \newtheorem{example}[definition]{Example}
  \newtheorem{remark}[definition]{Remark}
  \theoremstyle{plain}
  \newtheorem{lemma}[definition]{Lemma}
  \newtheorem{proposition}[definition]{Proposition}
  
  \newtheorem{corollary}[definition]{Corollary}
  
  \newcommand\Z{\mathbb{Z}}

\newcommand\g{\mathfrak{g}}
\renewcommand\a{\mathfrak{a}}
\newcommand\nm{\mathfrak{n}}

\makeatletter
\newcommand{\svast}{\bBigg@{3}}
\newcommand{\vast}{\bBigg@{4}}
\newcommand{\Vast}{\bBigg@{5}}
\makeatother\title[]{Closed subsets of root systems and regular subalgebras}

\begin{document}
\date{\today}

\author[Douglas]{Andrew Douglas$^{1, 2}$}
\address{$^1$Mathematics and Physics Programs, The Graduate Center, City University of New York, USA}
\address{$^2$Department of Mathematics, New York City College of Technology, City University of New York, USA}
\email{adouglas2@gc.cuny.edu}

\author[de Graaf]{Willem A. de Graaf$^{3}$}
\address{$^3$Department of Mathematics, University of Trento, Povo (Trento), Italy}
\email{degraaf@science.unitn.it}


\begin{abstract}
We describe an algorithm for classifying the closed subsets of a root system, up to conjugation by the associated Weyl group.  Such a classification for an irreducible root system is closely related to the classification of the regular subalgebras, up to inner automorphism, of the corresponding simple Lie algebra. We implement our algorithm to classify the closed subsets of the irreducible root systems of ranks $3$ through $7$. We present a complete description of the classification for the closed subsets of the rank $3$ irreducible root system. We employ this root system classification to classify all regular subalgebras of the rank $3$ simple Lie algebras.
We present only summary data for the classifications in higher ranks due to the large size of these classifications. Our algorithm is implemented in the language of the computer algebra system GAP.
\end{abstract}

\maketitle

\section{Introduction}

Closed subsets of root systems appear in a number of contexts. 
They are integral to the classification of regular semisimple subalgebras of
the exceptional Lie algebras \cite{dynkin, degraafa, lorent}.
Closed subsets of root systems also find application to classifications of the
reflection subgroups of finite and affine Weyl groups \cite{dyer}.  They
appear in the theory of Chevalley groups \cite{harebov}. And, the problem of
decomposing a root system into a union of two closed subsets has aided in the
classification of complex homogeneous spaces \cite{maly2}.

The problem of classifying all closed subsets of root systems is an unsolved problem, but incomplete results exist.  We mention three examples.
Borevich \cite{bor}
determined the number of topologies of a set of $n$ elements for $n\leq 8$.
Because topologies on a set of size $n$  correspond to closed sets in the root
system of type $A_{n-1}$ (see Section \ref{sec:top}), he therefore also
determined the number of closed sets in the root systems $A_l$ for $l\leq 7$.
Subsequently this has been taken further by other authors, see \cite{bm02} for
relatively recent work in this direction. Furthermore, in \cite{dch94} 
a classification of the invertible closed subsets of a root system was given.
Thirdly, the work of Dynkin \cite{dynkin} contains an algorithm to classify the
root subsystems of a root system, up to the action by the Weyl group.

In this article, we describe an algorithm for classifying the closed subsets of a root system, up to conjugation by the associated Weyl group.  Such a classification for an irreducible root system is closely related to the classification of the regular subalgebras, up to inner automorphism, of the corresponding simple Lie algebra. We implement our algorithm to classify the closed subsets of the irreducible root systems of ranks $3$ through $7$. We present a complete description of the classification for the closed subsets of the rank $3$ irreducible root systems. We employ this to classify all regular subalgebras of the rank $3$
simple Lie algebras. This is part of a bigger project to classify the
non semisimple subalgebras of small-dimensional simple Lie algebras,
see \cite{dc2, drsof, a2, c2, may}.
We present only summary data for the classifications in higher ranks due to the large size of these classifications. 

The article is organized as follows. In Section \ref{preliminaries}, we review root systems and closed subsets of root systems.
 In Section \ref{algorithm}, we describe our algorithm for classifying the closed subsets of root systems, along with its implementation in {\sf GAP}4
 \cite{GAP4}, using the {\sf SLA} package \cite{degraafa}. The latest
 version of this package (version 1.5) contains this implementation. We
 present the runtimes of the main steps of the algorithm on some sample
 inputs. All timings in this paper have been obtained on a 2.0Ghz processor with
 4GB of memory for {sf GAP}.

Next, in Section \ref{results}, we present the full classification of the closed subsets of the irreducible root system of rank $3$ derived from the implementation of our algorithm.   In this section we also present summary data for the classifications of  closed subsets of the irreducible root systems of ranks $4$ through $7$. 
In Section 
\ref{regular}, we discuss the relationship between  closed subsets of irreducible root systems and regular subalgebras of the corresponding simple Lie algebras. In this section, we also employ the classification of closed subsets of rank $3$ irreducible root systems to yield a classification of  regular subalgebras of the corresponding simple Lie algebras.

The ground field for Lie algebras in this article is the complex numbers. Unless explicitly stated otherwise, ``$\oplus$" denotes the direct sum of vector spaces,
not of Lie algebras.

{\bf Acknowledgement:} We thank Heiko Dietrich for his help with the
computation of the closed subsets of the root system of type $E_7$, for which
his computer worked several weeks.  AD would also like to thank WAdG and the Department of Mathematics at the University of Trento for their warm hospitality during a recent visit.

\section{Roots systems and their closed subsets}\label{preliminaries}

Let $\mathrm{E}$ be a Euclidean space with positive definite symmetric bilinear form denoted $(\alpha, \beta)$, for 
$\alpha, \beta \in \mathrm{E}$. A subset $\Phi$ of  $\mathrm{E}$  is called a {\it root system} in $\mathrm{E}$ if the following axioms 
are satisfied:
\vspace{2mm}

\noindent (R1) $\Phi$ is finite, spans $\mathrm{E}$, and does not contain $0$.\\
\noindent (R2) If $\alpha \in \Phi$, the only multiples of $\alpha$ in $\Phi$ are $\pm \alpha$.\\
\noindent (R3) If $\alpha \in \Phi$, the reflection  determined by $\alpha$ leaves $\Phi$ invariant.\\
\noindent (R4) If $\alpha, \beta \in \Phi$, then $\langle \beta,\alpha^\vee\rangle = \frac{2(\beta, \alpha)}{(\alpha,\alpha)} \in \mathbb{Z}$.

The root system $\Phi$ is called {\it irreducible} if it cannot be partitioned into the union of two
proper subsets such that each root in one set is orthogonal to each root in the other. Let $s_\alpha$ denote
the reflection determined by $\alpha\in \Phi$. The {\it Weyl group}
of a root system is the subgroup of $\mathrm{GL}(\mathrm{E})$ generated by the reflections $s_\alpha$, $\alpha\in \Phi$.


A subset $T$ of a root system $\Phi$ is {\it closed} if for any
$\alpha,\beta \in T$, $\alpha+\beta \in \Phi$ implies $\alpha+\beta \in T$. 
Any closed set $T$ is a disjoint union of its {\it symmetric} part $T^r =\{\alpha \in T | -\alpha \in T  \}$, 
and its {\it special} part $T^u=\{ \alpha \in T | -\alpha  \notin T \}$.

Closed subsets $T$ and $T'$ of a root system $\Phi$ are {\it conjugate} if there exists an element $w\in W(\Phi)$ such that
$w(T)=T'$. The following lemma is well-known (and easy to prove). 

\begin{lemma}\label{lem:deccl}
$T^r$ is a closed root subsystem of $\Phi$. For any two roots $\alpha \in T^u$ and $\beta \in T$ such that $\alpha+\beta$ is a root, we have that  $\alpha+\beta \in T^u$. In particular, $T^u$ is closed.
\end{lemma}

\subsection{Topologies on finite sets and closed subsets of the root systems of
type $A_{n-1}$}\label{sec:top}

Here we briefly describe the correspondence that exists between the
closed sets in the root systems of type $A_{n-1}$ and the topologies that a set
of size $n$ can have. To the best of our knowledge this correspondence has first
been described by Borevich (\cite{bor2}, \cite{bor}). Various computer
programs have been developed to enumerate the topologies of a finite set
(see for example \cite{bm02}). 

Let $n\geq 1$ and $X=\{1,2,\ldots,n\}$. A topology on $X$ is a collection
$\tau$ of subsets of $X$ such that $\emptyset, X\in \tau$ and the
union and intersection of elements of $\tau$ is again in $\tau$. The elements
of $\tau$ are said to be {\em closed sets}. For $i\in X$ we denote by
$\bar i$ the smallest closed set containing $i$.
The topology $\tau$ is called $T_0$ if for $i,j\in X$, $i\neq j$ we have that
$i\not\in \bar j$ or $j\not \in \bar i$. Two topologies $\tau,\tau'$ are
called {\em homeomorphic} if there is a $\pi\in S_n$ with
$\tau' = \{ \pi(T) : T \in \tau\}$.  

To a topology $\tau$ on $X$ we associate an $n\times n$-matrix $\sigma^\tau$,
where $\sigma^\tau_{ij}=1$ if $i\in \bar j$, and $\sigma^\tau_{ij}=0$ otherwise.
It is clear that from $\sigma^\tau$ we can reconstruct $\tau$. Furthermore,
a matrix $\sigma$ with coefficients $0,1$ corresponds to a topology if
\begin{equation}\label{top:1}
  \sigma_{ii}=1 \text{ for all } i \text{ and } \sigma_{ij}=1, \sigma_{jk}=1
  \text{ imply }  \sigma_{ik}=1.
\end{equation}  
  Finally, we have that $\tau,\tau'$ are homeomorphic if
and only if there is a $\pi\in S_n$ with $\sigma^{\tau'}_{ij} = \sigma^\tau_
{\pi(i),\pi(j)}$.

The root system $\Phi$ of type $A_{n-1}$ consists of all $\alpha_{ij}$ for
$1\leq i \neq j \leq n$. Here $\alpha_{1,2},\alpha_{2,3},\ldots,\alpha_{n-1,n}$
are the simple roots. If $i<j$ then $\alpha_{ij} = \alpha_{i,i+1}+
\alpha_{i+1,i+2}+\cdots +\alpha_{j-1,j}$. If $i>j$ then $\alpha_{ij} =
-\alpha_{ji}$. This implies that $\alpha_{ij}+\alpha_{kl}\in \Phi$
if $j=k$, when the sum is $\alpha_{il}$, or if $l=i$, when the sum
is $\alpha_{kj}$. The Weyl group of $\Phi$ is isomorphic to $S_n$,
and for $\pi\in S_n$ we have $\pi\cdot \alpha_{ij} = \alpha_{\pi(i),\pi(j)}$.

Let $\tau$ be a topology on $X$, and set $S=\{\alpha_{ij}\in \Phi \mid
\sigma^\tau_{ij}=1\}$. Then by \eqref{top:1} we see that $S$ is closed.
Moreover, $\tau$ is $T_0$ if and only if for each $i,j\in X$, $i\neq j$,
at least one of $\sigma^\tau_{ij}$, $\sigma^{\tau}_{ji}$ is 0. But this is
the same as saying that $S$ is special.

Conversely, let $S\subset \Phi$ be closed. Define the matrix $\sigma$
such that $\sigma_{ii}=1$ for all $i$ and $\sigma_{ij}=1$ if $\alpha_{ij}\in S$,
and $\sigma_{ij}=0$ otherwise. Then $\sigma$ satisfies \eqref{top:1},
and therefore defines a topology on $X$. It is also clear that two topologies
are homeomorphic if and only if the corresponding closed sets are conjugate
under $S_n$.

The conclusion is that topologies on $X$ correspond to closed sets in
$\Phi$, whereas topologies up to homeomorphism correspond to closed
sets up to conjugacy by the Weyl group. In this correspondence the
$T_0$ topologies correspond to special closed sets. Therefore, the tables in
\cite{bm02} imply that the root system of type $A_{15}$ contains 
4483130665195086 special closed sets.

\section{The algorithm}\label{algorithm}

In this section we outline an algorithm for listing the closed sets of
a root system $\Phi$. This algorithm is split into two parts. In the first
part we find all special closed subsets of $\Phi$ up to conjugacy. The second
part is a procedure to list, up to conjugacy, all non-special closed sets whose
special part is a given special closed set obtained by the first half of the
algorithm.

\subsection{Listing the special closed sets}\label{sec:3.1}

We first give a simple algorithm for listing the special closed sets of a
root system up to conjugacy by the Weyl group. Subsequently we describe some
ideas that can be used to implement it reasonably efficiently. 

The algorithm is based on three lemmas, the first of which 
is proved in  \cite{bou}, Chapter VI, Proposition 22,
see \cite{sopkina} for an easier proof.

\begin{lemma}\label{lem:1}
A special closed set is conjugate to a subset of $\Phi^+$.
\end{lemma}

For a closed set $T$ we let $T+T$ be the set of all $\gamma\in T$ such that
$\gamma=\alpha+\beta$ for certain $\alpha,\beta\in T$.
Let $\Phi^+$ be a fixed subset of positive roots with set of simple roots
$\Delta$. If $\alpha = \sum_{\alpha\in \Delta} k_\alpha \alpha$, then
$\sum_{\alpha\in\Delta} k_\alpha$ is called the {\em height} of $\alpha$.
Write $W=W(\Phi)$, the Weyl group of $\Phi$.

\begin{lemma}\label{lem:2}
Let $T$ be a special closed set. Then $T+T$ is strictly contained in $T$.
\end{lemma}

\begin{proof}
  Suppose that $T\subset \Phi^+$ and let $\alpha\in T$ be of minimal height.
  Then $\alpha$ cannot be a sum of two elements of $T$. The general case
  follows by Lemma \ref{lem:1}.
\end{proof}

For a closed set $T$ we denote its conjugacy class by $[T]$, that is,
$$[T] = \{ w(T) \mid w\in W\}.$$
Let $T$ be
a special closed set. By Lemma \ref{lem:2} there is an $\alpha \in T\setminus
(T+T)$. Set $T' = T\setminus \{\alpha\}$. Then $T'$ is closed and we say that
$[T']$ is a {\em successor} of $[T]$. Note that the set of successors of $[T]$
depends on the class $[T]$ and not on the choice of $T$. Indeed, let $S\in [T]$,
then there is a $w\in W$ with $w(T)=S$. So $\beta =w(\alpha) \in S\setminus
(S+S)$ and with $S'= S\setminus \{\beta\}$ we have $S' = w(T')$ and
therefore $[S']=[T']$.

\begin{lemma}\label{lem:3}
  Let $T'$ be a special closed set of cardinality $n<|\Phi^+|$.
  Then there is a
  special closed set $T$ of cardinality $n+1$ such that $[T']$ is a
  successor of $[T]$.
\end{lemma}

\begin{proof}
  By Lemma \ref{lem:1} we may assume that $T'\subset \Phi^+$. Set
  $$\mathcal{N}(T') = \{ \alpha\in \Phi \mid \alpha+\beta\in T' \text{ for all }
  \beta\in T' \text{ such that } \alpha+\beta\in \Phi\},$$
  which we call the {\em normalizer} of $T'$. We claim that there is a $\gamma
  \in \Phi^+$ with $\gamma\not\in T'$, $\gamma\in \mathcal{N}(T')$. 
  Let $\g$ be the semisimple Lie algebra over $\mathbb{C}$ with Cartan
  subalgebra $\mathfrak{h}$ such that the corresponding root system is
  isomorphic to $\Phi$. In the sequel we identify this root system and $\Phi$.
  Let $\nm^+$ be the subalgebra of $\g$ spanned by the root
  vectors corresponding to the roots in $\Phi^+$. Let $\a$ be the subspace
  of $\nm^+$ spanned by the root vectors corresponding to the roots in $T'$.
  Since $T'$ is closed, $\a$ is a subalgebra. Because $\a$ is spanned by root
  vectors we have $[\h,\a]\subset \a$. 
  Consider the normalizer of
  $\a$ in $\nm^+$, $\mathfrak{n}(\a) = \{ x \in \nm^+ \mid [x,\a]\subset \a\}$. 
  Let $x\in \mathfrak{n}(\a)$ and $y\in \a$,
  $h\in \mathfrak{h}$, then $[[h,x],y] = -[[x,y],h]-[[y,h],x]$. Now
  $[x,y]\in\a$ and hence $[[x,y],h]\in \a$. Secondly, $[y,h]\in \a$ so that
  $[[y,h],x]\in \a$ as well. We conclude 
  that $[h,x]\in \mathfrak{n}(\a)$. It follows that $\mathfrak{n}(\a)$ is
  a subalgebra of $\nm^+$, normalized by $\h$. Therefore, $\mathfrak{n}(\a)$ is
  also spanned by root vectors. If $\mathfrak{n}(\a)=\a$ then $\a$ would be
  a Cartan subalgebra of $\nm^+$. But the only Cartan subalgebra of $\nm^+$ is
  $\nm^+$ itself. We conclude that $\mathfrak{n}(\a)$ contains a root vector
  corresponding to a $\gamma\in \Phi^+$, not lying in $\a$. Then $\gamma\not\in
  T'$, $\gamma\in \mathcal{N}(T')$.

  Set $T = T'\cup \{\gamma\}$. Then $T$ is closed and
  $\gamma \in T\setminus (T+T)$. The lemma
  follows.
\end{proof}

\begin{remark}\label{rem:conjclos}
Let $T$ be a special closed set and $W_T = \{ w\in W \mid w(T) = T\}$, the
stabilizer of $T$ in $W$. Let $\alpha,\beta\in T\setminus (T+T)$ be such that
$\beta = w(\alpha)$, for a certain $w\in W_T$. Then $w(T\setminus \{\alpha\})
= T\setminus \{\beta\}$.
\end{remark}

This, together with the previous lemma, implies the correctness of the
following algorithm. It takes as input
a set $M$ of representatives of the conjugacy classes of special closed sets
of cardinality $n+1$. It outputs a set of representatives of the conjugacy
classes of special closed sets of cardinality $n$. The basic idea is to
compute the successors of each element of $M$ and weed out the conjugate
copies. Remark \ref{rem:conjclos} helps with the latter task.
The algorithm takes the following steps:

\begin{enumerate}
\item Set $L=\emptyset$.  
\item For each $T$ in $M$ do
  \begin{enumerate}
  \item Compute $W_T$.
  \item Compute the orbits of $W_T$ on $T\setminus (T+T)$.
  \item For each such orbit take a representative $\alpha$ and set $S=T\setminus
    \{\alpha\}$. If $S$ is not conjugate to any element of $L$, then add $S$
    to $L$.
  \end{enumerate}
\item Return $L$.
\end{enumerate}

Write $m=|\Phi^+|$. 
Now starting with the set $M$ containing the one special closed subset of
$\Phi^+$ of maximal cardinality, that is, $\Phi^+$ itself, we iterate this
algorithm and obtain the special closed sets (up to conjugacy) of sizes
$m-1,m-2,\ldots, 1$. So in the end we obtain, by Lemma \ref{lem:1}, all
special closed sets of $\Phi$ up to conjugacy.

\begin{remark}
We give some technical details of the implementation.  
The algorithm has been implemented in the language of the computer algebra
system {\sf GAP}4 \cite{GAP4}, as part of the package {\sf SLA} \cite{degraafb}.

In our implementation we use a list $\alpha_1,\ldots,\alpha_N$
of all the roots of the root system. A root is represented by its position
in this list. So instead of $\alpha_i$ we just use $i$. This gives a memory
efficient way to represent a closed set as a list of indices.
The addition relations of the roots
are stored in a table. This is a 2-dimensional table $\mathcal{T}$ such that
if $\mathcal{T}(i,j)=k$, then $\alpha_i+\alpha_j=\alpha_k$ if $k\neq 0$, and
$\alpha_i+\alpha_j$ is not a root if $k=0$.

The Weyl group is given as a
permutation group on $\{1,\ldots,N\}$. This allows us to use the functionality
for permutation groups present in {\sf GAP}4. Thus we can readily
compute generators of $W_T$ (the command for this is
{\tt Stabilizer( W, T, OnSets )}) and the orbits of $W_T$ on
$T\setminus (T+T)$ (the command here is {\tt Orbits( H, D )}, where
{\tt H}=$W_T$ and {\tt D }= $T\setminus (T+T)$).

For deciding whether two sets are conjugate under the Weyl group 
backtrack methods (see \cite{seress}, Chapter 9) are used. They are present
in the {\sf GAP} core system, the command for checking whether $S_1$,
$S_2$ are conjugate is 
{\tt RepresentativeAction( W, S1, S2, OnSets )}. 
\end{remark}

In order to make the implementation more efficient we make use of certain
invariants of closed sets. An invariant of closed sets is a function
$\chi : \{ \text{closed sets} \} \to A$, where $A$ is some set, such that
$S=w(T)$ for a $w\in W$ implies that $\chi(S)=\chi(T)$. First we describe
three invariants.
In order to describe the first two of these we need some well known facts
on weights (see \cite{humphreys}, Chapter 13), which we recall first.

Let $\Delta = \{ \alpha_1,\ldots,\alpha_\ell\}$ be a fixed set of simple roots
of $\Phi$. Define the elements $\lambda_i$ of the real space $E$ spanned by
$\Phi$ by $\langle \lambda_i, \alpha_j^\vee\rangle=\delta_{i,j}$. These
$\lambda_i$ are called the fundamental weights. A weight is a $\Z$-linear
combination of the fundamental weights. A weight $\lambda = m_1\lambda_1+\cdots
+m_\ell \lambda_\ell$ is {\em dominant} if $m_i\geq 0$ for all $i$. The Weyl group
$W$ acts on the set of weights, and each weight $\lambda$ is $W$-conjugate to a
unique dominant weight (\cite{humphreys}, Lemma 13.2.A). This dominant weight
is computed as follows. Set $\mu_1 = \lambda$. Let $i\geq 1$ and suppose that
$\mu_i$ is defined. Write $\mu_i = n_1\lambda_1 +\cdots +n_\ell \lambda_\ell$.
If $n_j\geq 0$ for all $j$ then $\mu_i$ is the dominant weight that we
are after. Otherwise there is a $j$ with $n_j <0$ and we set $\mu_{i+1} =
s_{\alpha_j}(\mu_i)$. It is not difficult to show that this procedure
terminates and returns the unique dominant weight conjugate to $\lambda$
(see, for example, \cite{degraaf0}, p. 319).
Note that by recording the $s_{\alpha_j}$ that are used we also find a
$w\in W$ such that $w(\lambda)$ is dominant.

Now let $S\subset \Phi$ then its sum vector is $\xi(S) = \sum_{\alpha\in S} \alpha$.
The sum vectors of conjugate sets are obviously conjugate, but in general not
the same. In order to get an invariant we compute the unique dominant
weight that is conjugate to $\xi(S)$. We denote the result by $\sigma(S)$. Then
$\sigma$ is an invariant of closed sets.

For a second invariant let $r$ be a rational number and 
consider the set $S^{(r)} = \{ \alpha\in \Phi \mid
(\alpha,\beta) \geq r \text{ for all } \beta\in S\}$. If $T=w(S)$ for a $w\in W$
then also $T^{(r)} = w(S^{(r)})$ because the Weyl group leaves the inner product
invariant. Define $\delta^r(S) = \sigma( S^r )$; then $\delta^r$ is an invariant
of closed sets. In our implementation we normalize the inner product so that
$(\alpha,\alpha)=2$ for short roots $\alpha$ and we use the invariant
$\delta^{-1}$. We also considered other values of $r$, but $r=-1$ turned out to
work best.

Let $T=\{\alpha_{i_1},\ldots,\alpha_{i_m}\}$ be a closed set. The matrix of its
inner products is $M_T = ((\alpha_{i_k},\alpha_{i_l}))_{1\leq k,l\leq m}$, and let
$\widehat{M}_T$ be the matrix consisting of the sorted rows of $M_T$, which
are listed in lexicographical order.
Since the Weyl group leaves the inner product invariant, if $T$, $S$ are
conjugate, then $M_T$ can be obtained from $M_S$ by applying a permutation
to both the rows and the columns. So in that case $\widehat{M}_T =
\widehat{M}_S$. Hence we have a third invariant.

We have three applications of these invariants: to make the computation of
the stabilizer $W_T$ more efficient, to reduce the number of sets $S$ that we
attempt to insert in the list $L$ in Step (2c), and reduce the number of
conjugacy tests that we perform in Step (2c) (that is, the number of calls
to {\tt RepresentativeAction}).

{\bf The computation of the stabilizer.} We can use
the invariant $\sigma(T)$ to compute a subgroup $W'\subset W$ such that the
$W'_T=W_T$, making the computation of the stabilizer more efficient. For this
we consider the sum vector $\xi(T)$. We compute a $w\in W$ such that
$w(\xi(T))$ is dominant. Let $\lambda_1,\ldots,\lambda_\ell$ denote the
fundamental weights as above and write $w(\xi(T)) = m_1\lambda_1+\cdots +m_\ell
\lambda_\ell$. Then the stabilizer $W_{w(\xi(T))}$ is generated by the
$s_{\alpha_i}$ where $i$ is such that
$m_i=0$ (\cite{humphreys}, Lemma 10.3B). Obviously $W_{\xi(T)} = w^{-1}
W_{w(\xi(T))}w$. We see that it is straightforward to compute the group
$W_{\xi(T)}$ and we can let this be $W'$. Furthermore, in practice it is
useful to check whether the generators of $W'$ leave $T$ invariant.
If this is the case then we have the stabilizer without any further
computation.

{\bf Reducing the number of sets $S$ in Step (2c).}
Note that the algorithm starts with the closed set $\Phi^+$ and the other closed
sets are constructed as subsets of that. Hence each closed set produced by
the algorithm is contained in $\Phi^+$.

In Step (2c) of the algorithm different non-conjugate sets $T$ of size $n+1$
can produce
conjugate successors $S$ of size $n$ (and usually this happens rather often).
In order to reduce the number of closed sets $S$ that go into the second part
of the step (that is, the conjugacy testing) we first compute all sets
$T'\subset \Phi^+$ having $S$ as a successor,
and only continue with $S$ if $\sigma(T)$
is minimal among all $\sigma(T')$. (For this we can use any order on the
weights, for example the lexicographical order on coordinate vectors.)
If it is not, then we are certain that a set conjugate to $S$ will be
constructed as successor of another set of size $n+1$. More precisely
we replace Step (2c) by the following

\begin{itemize}
\item[(2c)] For each such orbit take a representative $\alpha$ and set
  $S=T\setminus\{\alpha\}$.
  \begin{enumerate}[{\normalfont (i)}]
  \item Compute $\mathcal{N}^+(S) = \{\beta\in \Phi^+ \mid \beta\not\in S,
  S\cup \{\beta\} \text{ is closed } \}$.
\item For each $\beta\in \mathcal{N}^+(S)$ set $T' = S\cup \{\beta\}$ and
  compute $\sigma(T')$.
\item If $\sigma(T)$ is minimal among all $\sigma(T')$ then check whether
  $S$ is not conjugate to any element of $L$. If it is not then add $S$
  to $L$.
  \end{enumerate}
\end{itemize}

{\bf Reduce the number of conjugacy tests.}
Next we come to the conjugacy testing. 
In order to limit the necessary tests
we use the invariants $\widehat{M}_S$ and $\delta^{-1}(S)$. Only if a set $S'$
in the list $L$ has $\widehat{M}_S = \widehat{M}_{S'}$ and $\delta^{-1}(S) =
\delta^{-1}(S')$ we do a conjugacy test, otherwise we already know that $S,S'$
are not conjugate. Often no $S'$ in $L$ has $\widehat{M}_S = \widehat{M}_{S'}$
and $\delta^{-1}(S) =\delta^{-1}(S')$ and in those cases $S$ can directly
be added to $L$. This explains the relatively low number of conjugacy tests
in the sixth column of Table \ref{table:performance}.

\begin{remark}
\begin{itemize}  
\item  One other powerful invariant is the permanent of the
  matrix $M_T$. Experiments show that it is very effective. However, it is
  difficult
  to compute the permanent of a matrix, and the time needed to compute the
  permanents outweighs the gains made from using this invariant.
\item Of course it is possible to use the invariants $\delta^{-1}$,
  $\widehat{M}_T$ also for reducing the number of sets $S$ in Step (2c)(ii).
  In fact, experiments show that these manage to reduce the number of sets
  even further. However, the number of sets $T'$ for which these invariants
  have to be computed can be very large (see the fourth column of
  Tabe \ref{table:performance}). Therefore
  the invariant $\sigma$ works better because it is much easier to compute.
  Similarly, $\sigma$ could also be used to limit the number of conjugacy tests.
  But here the extra time needed to compute the invariants $\widehat{M}_T$
  and $\delta^{-1}$ pays off: these invariants limit the number of conjugacy
  tests more effectively.
\item Many alternatives to our algorithm are possible. Here we briefly go
  into two of them.
  Instead of our approach to conjugacy testing it is possible 
  compute a minimal image of each closed set; that is, a closed set $T$
  is replaced by the minimal element of its orbit (in some ordering, for
  example the the lexicographical). Then the conjugacy problem boils down to
  testing equality. In \cite{jjpw1} algorithms for computing minimal
  images are given which
  have been implemented in the {\sf GAP}4 package {\sf images} (\cite{jjpw2}).
  However, it appears that computing a minimal image of each closed set in the
  output of our algorithm takes longer than listing the closed sets in the
  first place (for example, for the root system of type $E_6$ we find
  94635 special closed sets which are listed in 149 seconds, whereas the
  computation of a minimal image of each of the 94635 closed sets took
  443 seconds).
  Another more general purpose algorithm that can be used for our problem is
  described in \cite{bet}. The rough structure of this algorithm is similar
  to ours. It starts with two sets $\mathcal{A}$ and $\mathcal{B}$ (in our
  case these are the special closed sets of sizes $n+1$ and $n$ respectively)
  on which a group $G$ acts (in our case the Weyl group). Also given is a
  relation $\mathcal{R}\subset \mathcal{A}\times \mathcal{B}$ (in our case
  this is the successor relation). Knowing the $G$-orbits in $\mathcal{A}$
  the algorithm computes the $G$-orbits in $\mathcal{B}$. We have not
  implemented this algorithm, but it seems to be hard to combine it 
  with the invariants that we use and which are very succesful at making our
  algorithm more efficient. Probably it would be possible to modify the
  algorithm in \cite{bet} to also use those invariants, but then the result
  would be a method very similar to ours.
\item There always is the obvious question as to the reliability of our
  implementations. In other words: how can we be sure that output of our
  computations is correct. Of course we can never be absolutely sure.
  However, we do have some circumstantial evidence that our implementations
  are correct. Firstly, in the course of our work we also have tried
  completely different algorithms, and the output has always been the same.
  Secondly, for the root systems of type $A_n$ the number of special
  closed sets that we get is confirmed by the data in \cite{bm02}.
\end{itemize}  
\end{remark}

\subsection{Practical performance}

Here we present some data on the practical performance of our algorithm.
In Table \ref{table:performance} we list the running times of various parts
of the algorithm. As test root systems we have taken $F_4$, $E_6$, $B_6$,
$B_7$. The second column of this table has the total time needed to compute
the stabilizers $W_T$. The third column lists the time taken to compute
the orbits of $W_T$ on $T\setminus (T+T)$. The fourth column has the time taken
to compute $\sigma(T)$ for the various $T$ (to limit the number of sets $S$
that we try to insert in the list $L$). The fifth column displays the total
time taken to compute the invariants $\widehat{M}_T$, $\delta^{-1}$. The
sixth column has the total time taken to do the conjugacy tests (that is,
the time taken in calls to {\tt RepresentativeAction}). The integer
following the running times in these columns is the number of times the
algorithm entered the computation (so, for example, for $B_6$ there have been
1273371 computations of $\sigma(T)$). The seventh and eighth colums show,
respectively, the number of closed sets in the output and the total time
taken. The last column has the number of miliseconds taken for each closed set
in the output (this is the quotient of the entries in the eighth and seventh
columns).

We list a few obvious observations:

\begin{itemize}
\item The parts contributing most to the running time are the computation of
  the stabilizers (which is done for almost every final closed set), and the
  computation of the invariant $\sigma$ (which is done a huge number of times).
\item The conjugacy test is not perfomed very often (compared to the number
  of closed sets in the output).
\item The time taken per closed set increases markedly with increasing rank.  
\end{itemize}

\begin{footnotesize}
\renewcommand{\arraystretch}{1.4}
\begin{table}[h!]
\begin{tabular}{|c|c|c|c|c|c|c|c|c|}
 \hline 
 \rowcolor{Gray}

 {\bf type}  & {\bf Stab } & {\bf Orb }  & {\bf Invariant} &
 {\bf Invariants  } & {\bf Conjugacy} & {\bf Number of} & {\bf Time} &
 {\bf ms per}\\
  \rowcolor{Gray}
  & & & $\sigma$ & $\widehat{M}_T$, $\delta^{-1}$ & & {\bf closed sets} & &
           {\bf closed set}\\
 \hline 
$F_4$ & 0.9 & 0.3 & 0.4 & 0.4 & 0.12 & 3579 & 2.2 & 0.61\\
 & 3577 & 3577 & 19889 & 3954 & 410 & & & \\
 \hline
$E_6$ & 56.7 & 9.6 & 36.6 & 24.2 & 17.8 & 94635 & 149 & 1.57\\ 
 & 94634 & 94634 & 828924 & 115652 & 30925 & & & \\
 \hline
 $B_6$ & 90.2 & 16.9 & 50.5 & 23.5 & 13.4 & 151386 & 201 & 1.33\\
 & 151384 & 151384 & 1273371 & 174243 & 27234 & & & \\
 \hline
  $B_7$ & 5009 & 659 & 4656 & 1506 & 1414 & 6115473 & 13803 & 2.26\\
 & 6115471 & 6115471 & 69416764 & 7223257 & 1718047 & & & \\
 \hline
\end{tabular}
\caption{Running times (in seconds)
  of various parts of the algorithm. See the text for
  a description of the contents of the columns.}\label{table:performance}
\end{table}
\end{footnotesize}

\subsection{Listing the non-special closed sets}

Now we turn to the second part of the procedure, that is, to finding the
non-special closed sets, given the special ones. We will outline an algorithm
for obtaining the closed sets $Q$ such that $Q^u=T$, where $T$ is a given
special closed set. Since we have the list of all special closed sets, by
Lemma \ref{lem:deccl} this
suffices to get all closed sets. We start with the following simple
lemma.

\begin{lemma}\label{lem:4}
  Let $T$ be a special closed set.
  Let $\alpha \in \Phi\setminus (T\cup -T)$. There is a closed set $S$ such that
  $S^u=T$ and $\alpha,-\alpha \in S^r$ if and only if 
  $T\cup \{\alpha,-\alpha\}$ is closed. 
\end{lemma}

\begin{proof}
  If there is such an $S$, then Lemma \ref{lem:deccl} shows that 
  $T\cup \{\alpha,-\alpha\}$ is closed. For the converse we can take
  $S=T\cup \{\alpha,-\alpha\}$.
\end{proof}

The next lemma is similar to \cite{sopkina}, Proposition 5.

\begin{lemma}\label{lem:5}
  Let $T$ be a special closed set.
  Set $S = \{ \alpha \in \Phi\setminus (T\cup -T)\mid 
  T\cup \{\alpha,-\alpha\} \text{ is closed}\}$. Then 
  $R= S \cup T$ is a closed set with $R^r= S$, $R^u=T$.
\end{lemma}

\begin{proof}
  We first show that $S$ is closed. Let $\alpha,\beta\in S$ be such that
  $\alpha+\beta\in \Phi$. We show that $T\cup \{\alpha+\beta,-\alpha-\beta\}$
  is closed. 
  Let $\gamma\in T$ be such that $\alpha+\beta+\gamma
  \in \Phi$. Then at least two of $\alpha+\beta$, $\alpha+\gamma$,
  $\beta+\gamma$ are roots. (If $\Phi$ is simply laced, this can be shown
  by using the property that for $\delta,\epsilon\in \Phi$ we have that
  $\delta+\epsilon\in \Phi$ if and only if $(\delta,\epsilon)<0$. In general
  the intersection of $\Phi$ and the set of linear combinations of $\alpha$,
  $\beta$, $\gamma$ is a root system of rank at most 3. For those root systems
  it is easy to show the statement by a brute force verification.) So at
  least one of $\alpha+\gamma$, $\beta+\gamma$ lies in $\Phi$. Suppose
  $\alpha+\gamma\in \Phi$. Since $T\cup \{\alpha,-\alpha\}$ is closed,
  $\alpha+\gamma\in T$. But then also $\beta+
  \alpha+\gamma\in T$. In the same way we see that if $-\alpha-\beta+\gamma\in
  \Phi$ then $-\alpha-\beta+\gamma\in T$. Hence $T\cup
  \{\alpha+\beta,-\alpha-\beta\}$ is closed and the same follows for $S$.

  Now it is obvious that $R$ is closed. 
  Furthermore, if $\alpha\in S$ then also $-\alpha\in S$ so that $R^r= S$,
  implying also $R^u=T$.
\end{proof}

Now fix a special closed set $T$.
The algorithm first constructs the set $S$ as in the previous lemma.
By Lemma \ref{lem:4}, $S$ contains all closed sets that are the symmetric
part of a closed set whose special part is $T$. So an immediate way of
proceeding would be to enumerate all symmetric closed subsets of $S$, for
each such set $P$ form the closed set $P\cup T$, and get rid of conjugate
copies. However, there is a much more efficient way of doing it.

First note that $S$ is a root subsystem of $\Phi$. Let $\alpha\in S$, then
the reflection $s_\alpha$ sends $T$ to itself. (Indeed, if $\beta\in T$ then the
string $\{ \beta+k\alpha \mid k\in \Z\}$ is contained in $S$ as $S$ is closed,
and it has to be contained in $T$ by Lemma \ref{lem:deccl}.)
So $T$ is stabilized by the Weyl group $W^S$ of $S$.
Hence, if two closed symmetric subsets $P,P'\subset S$ are conjugate under 
$W^S$, then $P\cup T$, $P'\cup T$ are as well, so that they are
conjugate under $W$. Now we use Dynkin's algorithm (\cite{dynkin}, see
\cite{degraaf17}, \S 5.9 for a more recent description)
to find all symmetric closed
subsets of $S$ up to conjugacy by $W^S$. For each symmetric closed set $P$
found we form the set $P\cup T$. Let $W_T$ be the stabilizer of $T$ in
$W$. Two closed subsets $P\cup T$, $P'\cup T$ are conjugate under $W$ if
and only if they are under $W_T$. So we erase the $W_T$-conjugate copies
from the list and we are done. Note that in order to execute this we do
not have to compute the stabilizers $W_T$ as they have already been constructed
by the algorithm of Section \ref{sec:3.1}.

By comparing the running times for the root systems $F_4$, $B_5$, $B_6$, $E_6$
in Tables \ref{table:performance}, \ref{comparison} we see that computing the
non-special closed sets does not weigh heavily on the runnning time. The main
part of the computation time is taken for listing the special closed sets.

\section{Closed subsets of irreducible root systems}\label{results}

In this section, we present the full classification of the closed subsets of the irreducible root systems of rank $3$ derived from the implementation of our algorithm. This classification is contained in Tables \ref {a3resultsa} through \ref{b3resultsb}.   In this section we also present summary data for the classifications of  closed subsets of the irreducible root systems of ranks $4$ through $7$ in Table \ref{comparison}. 

For each root system,  
the classification is divided into three parts: The symmetric closed subsets $T^r$, for which the special part is empty $T^u=\emptyset$; the special closed subsets $T^u$, for which the symmetric part is empty $T^r=\emptyset$; and the Levi decomposable closed subsets $T^r\mathbin{\dot{\cup}} T^u$, for which
$T^r \neq \emptyset$ and  $T^u\neq \emptyset$.   

For symmetric closed subsets, we identify the isomorphism classes. 
The symmetric closed subsets have been computed with the implementation
 of Dynkin's algorithm (\cite{dynkin}) in the {\sf SLA} package
 (\cite{degraafb}) of {\sf GAP}4.

\begin{footnotesize}
\renewcommand{\arraystretch}{1.4}
\begin{table}[h!]
\begin{tabular}{|c|c|c|c|c|c|}
 \hline 
 \rowcolor{Gray}
& {\bf Special closed subsets}  $\mathbf{T^u}$ ($\mathbf{T^r= 0}$)\\
 \hline 
1& $\alpha_1, \alpha_2, \alpha_3, \alpha_1+\alpha_2, \alpha_2+\alpha_3, \alpha_1+\alpha_2+\alpha_3$\\
 \hline
 2& $\alpha_2,  \alpha_3, \alpha_1+\alpha_2, \alpha_2+\alpha_3, \alpha_1+\alpha_2+\alpha_3$\\
\hline
3& $\alpha_1,  \alpha_2, \alpha_1+\alpha_2, \alpha_2+\alpha_3, \alpha_1+\alpha_2+\alpha_3$\\
\hline
4& $\alpha_1,  \alpha_3, \alpha_1+\alpha_2, \alpha_2+\alpha_3, \alpha_1+\alpha_2+\alpha_3$\\
\hline
5 &$ \alpha_3, \alpha_1+\alpha_2, \alpha_2+\alpha_3, \alpha_1+\alpha_2+\alpha_3$\\
\hline
6 &$ \alpha_1, \alpha_1+\alpha_2, \alpha_2+\alpha_3, \alpha_1+\alpha_2+\alpha_3$\\
\hline
7&$ \alpha_2, \alpha_1+\alpha_2, \alpha_2+\alpha_3, \alpha_1+\alpha_2+\alpha_3$\\
\hline
\end{tabular}
\vspace{1mm}
\caption{The special and Levi decomposable closed subsets of the root system $A_3$.}\label{a3resultsa}
\end{table}
\end{footnotesize}

\renewcommand{\arraystretch}{1.4}
\begin{table}[h!]
\begin{tabular}{|c|c|c|c|c|c|}
  \hline 
   \rowcolor{Gray}
&{\bf Symmetric closed subsets} $\mathbf{T^r}$ ($\mathbf{T^u= 0}$) & {\bf Isomorphism class}\\
\hline
29 &$\alpha_1$, $-\alpha_1$& $A_1$\\
 \hline 
30 & $\alpha_1, \alpha_2, \alpha_1+\alpha_2, -\alpha_1, -\alpha_2, -\alpha_1-\alpha_2$& $A_2$\\
 \hline 
 31 & $\alpha_1$, $\alpha_3$, $-\alpha_1$, $-\alpha_3$& $A_1\times A_1$\\
 \hline 
 32 & $\Phi$& $A_3$\\
 \hline 
\end{tabular}
\vspace{1mm}
\caption{The symmetric closed subsets of the root system $A_3$.}\label{a3resultsb}
\end{table}

\clearpage

\begin{tiny}
\renewcommand{\arraystretch}{1.32}
\begin{table}[h!]
\begin{tabular}{|c|c|c|c|c|c|}
 \hline 
  \rowcolor{Gray}
 &{\bf Special closed subsets}  $\mathbf{T^u}$ ($\mathbf{T^r= 0}$)\\
 \hline

1& $\Phi^+$\\
\hline
2& $
\alpha_1,   \alpha_2, 
 \alpha_1+\alpha_2, 
 \alpha_2+\alpha_3, 
 \alpha_1+\alpha_2+ \alpha_3,   \alpha_2+2\alpha_3$, $\alpha_1+\alpha_2+2 \alpha_3,  \alpha_1+2\alpha_2+2 \alpha_3$\\
\hline
3&   $\alpha_2, \alpha_3, 
 \alpha_1+\alpha_2, 
 \alpha_2+\alpha_3,  \alpha_1+\alpha_2+ \alpha_3,  \alpha_2+2\alpha_3,$
  $\alpha_1+\alpha_2+2 \alpha_3,  \alpha_1+2\alpha_2+2 \alpha_3$ \\
 \hline
4&
$ \alpha_1, \alpha_3, 
 \alpha_1+\alpha_2, 
 \alpha_2+\alpha_3,  \alpha_1+\alpha_2+ \alpha_3,  \alpha_2+2\alpha_3,$
  $\alpha_1+\alpha_2+2 \alpha_3,  \alpha_1+2\alpha_2+2 \alpha_3$\\
\hline
5&
$\alpha_1,  \alpha_2, 
 \alpha_1+\alpha_2, 
 \alpha_2+\alpha_3,  \alpha_1+\alpha_2+ \alpha_3,  \alpha_2+2\alpha_3,$ $\alpha_1+\alpha_2+2 \alpha_3,  \alpha_1+2\alpha_2+2 \alpha_3$\\
 \hline
6&
$\alpha_3, 
 \alpha_1+\alpha_2, 
 \alpha_2+\alpha_3,  \alpha_1+\alpha_2+ \alpha_3,  \alpha_2+2\alpha_3,$
   $\alpha_1+\alpha_2+2 \alpha_3,  \alpha_1+2\alpha_2+2 \alpha_3$\\
\hline
7&$\alpha_1, 
 \alpha_1+\alpha_2, 
 \alpha_2+\alpha_3,  \alpha_1+\alpha_2+ \alpha_3,  \alpha_2+2\alpha_3,$
 $\alpha_1+\alpha_2+2 \alpha_3,  \alpha_1+2\alpha_2+2 \alpha_3$
\\
\hline
\end{tabular}
\caption{The special closed subsets of the root system $B_3$.}\label{b3resultsa}
\end{table}
\end{tiny}

\begin{scriptsize}
\renewcommand{\arraystretch}{1.4}
\begin{table}[h!]
\begin{tabular}{|c|c|c|c|c|c|}
 \hline 
 \rowcolor{Gray}
&{\bf Levi decomposable closed subsets} $\mathbf{T^r \oplus T^u}$  ($\mathbf{T^r}, \mathbf{T^u\neq 0}$)\\  
 \hline 
47&
$\alpha_1,   \alpha_2, \alpha_3, 
 \alpha_1+\alpha_2, 
 \alpha_2+\alpha_3,  \alpha_1+\alpha_2+ \alpha_3,  \alpha_2+2\alpha_3,
 \alpha_1+\alpha_2+2 \alpha_3,  \alpha_1+2\alpha_2+2 \alpha_3, -\alpha_3$
\\
\hline
48&
$\alpha_1,  \alpha_2, \alpha_3, 
 \alpha_1+\alpha_2, 
 \alpha_2+\alpha_3,  \alpha_1+\alpha_2+ \alpha_3,  \alpha_2+2\alpha_3,
 \alpha_1+\alpha_2+2 \alpha_3,  \alpha_1+2\alpha_2+2 \alpha_3, -\alpha_1$

\\
\hline
49&
$\alpha_1,  \alpha_2, \alpha_3, 
 \alpha_1+\alpha_2, 
 \alpha_2+\alpha_3,  \alpha_1+\alpha_2+ \alpha_3,  \alpha_2+2\alpha_3,
 \alpha_1+\alpha_2+2 \alpha_3,  \alpha_1+2\alpha_2+2 \alpha_3, -\alpha_2$

\\
\hline
50&
$\alpha_1, \alpha_2, 
 \alpha_1+\alpha_2, 
 \alpha_2+\alpha_3,  \alpha_1+\alpha_2+ \alpha_3,  \alpha_2+2\alpha_3,
 \alpha_1+\alpha_2+2 \alpha_3,  \alpha_1+2\alpha_2+2 \alpha_3, -\alpha_1$

\\
\hline
51&$\alpha_2, \alpha_3, 
 \alpha_1+\alpha_2, 
 \alpha_2+\alpha_3,  \alpha_1+\alpha_2+ \alpha_3,  \alpha_2+2\alpha_3,
 \alpha_1+\alpha_2+2 \alpha_3,  \alpha_1+2\alpha_2+2 \alpha_3, -\alpha_3$

\\
\hline
52&
$\alpha_1, \alpha_2, \alpha_3, 
 \alpha_1+\alpha_2, 
 \alpha_2+\alpha_3,  \alpha_1+\alpha_2+ \alpha_3,  \alpha_2+2\alpha_3,
 \alpha_1+\alpha_2+2 \alpha_3,  \alpha_1+2\alpha_2+2 \alpha_3, -\alpha_1, -\alpha_3$

\\
\hline
53&$\alpha_1,  \alpha_2, 
 \alpha_1+\alpha_2, 
  \alpha_1+\alpha_2+\alpha_3,
 \alpha_2+2\alpha_3,
 \alpha_1+\alpha_2+2 \alpha_3,  \alpha_1+2\alpha_2+2 \alpha_3, -\alpha_2$

\\
\hline
\end{tabular}
\caption{The Levi decomposable closed subsets of the root system $B_3$.}\label{b3resultsa}
\end{table}
\end{scriptsize}

\begin{small}
\renewcommand{\arraystretch}{1.4}
\begin{table}[h!]
\begin{tabular}{|c|c|c|c|c|c|}
  \hline 
   \rowcolor{Gray}
&{\bf Symmetric closed subsets} $\mathbf{T^r}$ ($\mathbf{T^u= 0}$) & {\bf Isomorphism class}\\
\hline
80& $\alpha_1, -\alpha_1$ & $A_1$\\
\hline
81& $\alpha_3, -\alpha_3$ & $A_1$\\
\hline
82& $\alpha_1, \alpha_3, -\alpha_1, -\alpha_3$ &  $A_1\times A_1$\\
\hline
83& $\alpha_1, \alpha_1+2\alpha_2+2\alpha_3, -\alpha_1, -\alpha_1-2\alpha_2-2\alpha_3$ &  $A_1\times A_1$\\
\hline
84& $\alpha_1, \alpha_2, \alpha_1+\alpha_2, -\alpha_1, -\alpha_2,- \alpha_1-\alpha_2$ &  $A_2$\\
\hline 
85& $\alpha_1,  \alpha_2+\alpha_3, \alpha_1+ \alpha_2+\alpha_3, \alpha_1+ 2\alpha_2+2\alpha_3$, &  $B_2$\\
&    $-\alpha_1,  -\alpha_2-\alpha_3, -\alpha_1- \alpha_2-\alpha_3, -\alpha_1- 2\alpha_2-2\alpha_3$ &\\
\hline
86&  $\alpha_1, \alpha_3, \alpha_1+ 2\alpha_2+2\alpha_3$,  $-\alpha_1, -\alpha_3, -\alpha_1- 2\alpha_2-2\alpha_3$ &   $A_1\times A_1\times A_1$\\
\hline
87& $\alpha_1, \alpha_2, \alpha_1+\alpha_2, \alpha_2+2\alpha_3, \alpha_1+\alpha_2+2\alpha_3, \alpha_1+2\alpha_2+2\alpha_3$, & $A_3$\\
& $-\alpha_1, -\alpha_2, -\alpha_1-\alpha_2, -\alpha_2-2\alpha_3, -\alpha_1-\alpha_2-2\alpha_3, -\alpha_1-2\alpha_2-2\alpha_3$ &\\
\hline
 88 & $\Phi$& $B_3$\\
 \hline 
\end{tabular}
\vspace{1mm}
\caption{The symmetric closed subsets of the root system $B_3$.}\label{b3resultsb}
\end{table}
\end{small}

\clearpage

\begin{scriptsize}
\renewcommand{\arraystretch}{1.28}
\begin{table}[h!]
\begin{tabular}{|c|c|c|c|c|c|}
 \hline 
  \rowcolor{Gray}
& {\bf Special closed subsets}  $\mathbf{T^u}$ ($\mathbf{T^r= 0}$)\\
 \hline 
1&$\Phi^+$ \\
\hline

2&$\alpha_2,  \alpha_3,
 \alpha_1+\alpha_2, 
 \alpha_2+\alpha_3, 
 \alpha_1+\alpha_2+ \alpha_3,   2\alpha_2+\alpha_3,
 \alpha_1+2\alpha_2+ \alpha_3,  2\alpha_1+2\alpha_2+ \alpha_3$

 \\
 \hline
 3&$\alpha_1,  \alpha_2,
 \alpha_1+\alpha_2, 
 \alpha_2+\alpha_3, 
 \alpha_1+\alpha_2+ \alpha_3,   2\alpha_2+\alpha_3,
 \alpha_1+2\alpha_2+ \alpha_3,  2\alpha_1+2\alpha_2+ \alpha_3$

\\
\hline
4&
$\alpha_1,  \alpha_3,
 \alpha_1+\alpha_2, 
 \alpha_2+\alpha_3, 
 \alpha_1+\alpha_2+ \alpha_3,   2\alpha_2+\alpha_3,
 \alpha_1+2\alpha_2+ \alpha_3,  2\alpha_1+2\alpha_2+ \alpha_3$
 
 \\
 \hline
 5& $\alpha_3,
 \alpha_1+\alpha_2, 
 \alpha_2+\alpha_3, 
 \alpha_1+\alpha_2+ \alpha_3,   2\alpha_2+\alpha_3,
 \alpha_1+2\alpha_2+ \alpha_3,  2\alpha_1+2\alpha_2+ \alpha_3$
 
 \\
 \hline
 6& $\alpha_1,
 \alpha_1+\alpha_2, 
 \alpha_2+\alpha_3, 
 \alpha_1+\alpha_2+ \alpha_3,   2\alpha_2+\alpha_3,
 \alpha_1+2\alpha_2+ \alpha_3,  2\alpha_1+2\alpha_2+ \alpha_3$

 \\
 \hline
 7& $\alpha_1, \alpha_3,
 \alpha_1+\alpha_2, 
 \alpha_2+\alpha_3, 
 \alpha_1+\alpha_2+ \alpha_3,   2\alpha_2+\alpha_3,
 \alpha_1+2\alpha_2+ \alpha_3,  2\alpha_1+2\alpha_2+ \alpha_3$
 
 \\
\hline
\end{tabular}
\vspace{1mm}
\caption{The special closed subsets of the root system $C_3$.}\label{c3resultsa}
\end{table}
\end{scriptsize}

\begin{scriptsize}
\renewcommand{\arraystretch}{1.4}
\begin{table}[h!]
\begin{tabular}{|c|c|c|c|c|c|}
 \hline 
  \rowcolor{Gray}
&{\bf Levi decomposable closed subsets} $\mathbf{T^r \oplus T^u}$  ($\mathbf{T^r}, \mathbf{T^u\neq 0}$)\\  
 \hline  
45&
$\alpha_3, \alpha_1+ \alpha_2, \alpha_1+\alpha_2+ \alpha_3,  
 2\alpha_1+2\alpha_2+ \alpha_3, -\alpha_3$
 
 \\
 \hline
46&$ \alpha_1, \alpha_2,  \alpha_3,
 \alpha_1+\alpha_2, 
 \alpha_2+\alpha_3, 
 \alpha_1+\alpha_2+ \alpha_3,   2\alpha_2+\alpha_3,
 \alpha_1+2\alpha_2+ \alpha_3,  2\alpha_1+2\alpha_2+ \alpha_3, -\alpha_1$
 
  \\
 \hline
 47&
$\alpha_1, \alpha_2,  \alpha_3,
 \alpha_1+\alpha_2, 
 \alpha_2+\alpha_3, 
 \alpha_1+\alpha_2+ \alpha_3,   2\alpha_2+\alpha_3,
 \alpha_1+2\alpha_2+ \alpha_3,  2\alpha_1+2\alpha_2+ \alpha_3, -\alpha_3$

 \\
 \hline
 48&$\alpha_1, \alpha_2,  \alpha_3,
 \alpha_1+\alpha_2, 
 \alpha_2+\alpha_3, 
 \alpha_1+\alpha_2+ \alpha_3,   2\alpha_2+\alpha_3,
 \alpha_1+2\alpha_2+ \alpha_3,  2\alpha_1+2\alpha_2+ \alpha_3, -\alpha_2$

 \\
 \hline
 49&
$\alpha_1, \alpha_2,
 \alpha_1+\alpha_2, 
 \alpha_2+\alpha_3, 
 \alpha_1+\alpha_2+ \alpha_3,   2\alpha_2+\alpha_3,
 \alpha_1+2\alpha_2+ \alpha_3,  2\alpha_1+2\alpha_2+ \alpha_3, -\alpha_1$

 \\
 \hline
 50&
$\alpha_2, \alpha_3,
 \alpha_1+\alpha_2, 
 \alpha_2+\alpha_3, 
 \alpha_1+\alpha_2+ \alpha_3,   2\alpha_2+\alpha_3,
 \alpha_1+2\alpha_2+ \alpha_3,  2\alpha_1+2\alpha_2+ \alpha_3, -\alpha_3$
 
  \\
 \hline
 51&$\alpha_1, \alpha_2, \alpha_3,
 \alpha_1+\alpha_2, 
 \alpha_2+\alpha_3, 
 \alpha_1+\alpha_2+ \alpha_3,   2\alpha_2+\alpha_3,
 \alpha_1+2\alpha_2+ \alpha_3,  2\alpha_1+2\alpha_2+ \alpha_3, -\alpha_1, -\alpha_3$
 
  \\
\hline
\end{tabular}
\vspace{1mm}
\caption{The  Levi decomposable closed subsets of the root system $C_3$.}\label{c3resultsa}
\end{table}
\end{scriptsize}

\begin{small}
\renewcommand{\arraystretch}{1.4}
\begin{table}[h!]
\begin{tabular}{|c|c|c|c|c|c|}
  \hline 
   \rowcolor{Gray}
&{\bf Symmetric closed subsets} $\mathbf{T^r}$ ($\mathbf{T^u= 0}$) & {\bf Isomorphism class}\\
\hline
79&$\alpha_1, -\alpha_1$& $A_1$ \\ 
\hline
80&$\alpha_3, -\alpha_3$& $A_1$ \\ 
\hline
81&$\alpha_1, \alpha_3, -\alpha_1, -\alpha_3$ & $A_1 \times A_1$\\
\hline
82&$\alpha_3, 2\alpha_2+\alpha_3, -\alpha_3, -2\alpha_2-\alpha_3$ & $A_1 \times A_1$\\
\hline
83&$\alpha_1, \alpha_2, \alpha_1+\alpha_2, -\alpha_1, -\alpha_2, -\alpha_1-\alpha_2$& $A_2$\\
\hline
84&$\alpha_1, 2\alpha_2+\alpha_3, \alpha_1+2\alpha_2+\alpha_3,  2\alpha_1+2\alpha_2+\alpha_3$,& $B_2$\\
 &$-\alpha_1, -2\alpha_2-\alpha_3, -\alpha_1-2\alpha_2-\alpha_3, -2\alpha_1-2\alpha_2-\alpha_3$&\\
\hline
85&$\alpha_3, 2\alpha_2+\alpha_3, 2\alpha_1+2\alpha_2+\alpha_3, -\alpha_3, -2\alpha_2-\alpha_3, -2\alpha_1-2\alpha_2-\alpha_3$ & $A_1 \times A_1\times A_1$\\
\hline
86&$\alpha_1, \alpha_3, 2\alpha_2+\alpha_3, \alpha_1+2\alpha_2+\alpha_3,  2\alpha_1+2\alpha_2+\alpha_3$,& $A_1\times B_2$\\
 &$-\alpha_1, -\alpha_3, -2\alpha_2-\alpha_3, -\alpha_1-2\alpha_2-\alpha_3, -2\alpha_1-2\alpha_2-\alpha_3$&\\
\hline
 87 & $\Phi$& $C_3$\\
 \hline 
\end{tabular}
\vspace{1mm}
\caption{The symmetric closed subsets of the root system $C_3$. }\label{b3resultsb}
\end{table}
\end{small}

\clearpage


\begin{landscape}

\begin{footnotesize}
\renewcommand{\arraystretch}{1.4}
\begin{table}[h!]
\begin{tabular}{|c|c|c|c|c|c|c|}
 \hline 
 \rowcolor{Gray}
 {\bf Root system}  & {\bf Number of } &{\bf Number of }  &{\bf Number of }  &{\bf Total Number of }  & {\bf Time (s)} \\
  \rowcolor{Gray}
 &{\bf special closed subsets} &{\bf Levi decomposable }  &{\bf symmetric closed subsets }&{\bf closed subsets} & \\
   \rowcolor{Gray}
   &&{\bf closed subsets}&& & \\
 \hline 
$A_3$ &15 &13&4&32 & 0.02 \\  %
\hline
$B_3$ & 46 &33&9&88 & 0.05\\ 
\hline
$C_3$ &44&34&9& 87 & 0.05\\  %
\hline \hline
$A_4$ & 62&70&6&138 & 0.07\\
\hline
$B_4$ & 429&311&19&759 & 0.4\\
\hline
$C_4$ &401&334&19&754 & 0.4 \\
\hline
$D_4$ & 110&102&11&223 & 0.4 \\
\hline
$F_4$ &3579&1242&23& $4844$ & 3.0 \\
\hline \hline
$A_5$ &317&390&10& 717 & 0.4 \\
\hline
$B_5$ &6267&3592&35&9894 & 7.6 \\
\hline
$C_5$ &5744&4074&35& 9853 & 7.7 \\
\hline
$D_5$ &1145&877&15&2037 & 1.4 \\
\hline \hline
$A_6$ &2044&2476&14& 4534 & 3.3 \\
\hline
$B_6$ &151386&61001&64& 212451 & 257 \\
\hline
$C_6$ &137456&74081&64& 211601 & 239\\
\hline
$D_6$ &20549&11310&31& 31890 & 36 \\
\hline
$E_6$ &94635&29487&20& 124142 & 172 \\
\hline
\hline
$A_7$ & 16998 & 18959 & 21 & 35978 & 33 \\
\hline
$B_7$ & 6115473 & 1648192 & 109 & 7763774 & 15569 \\
\hline
$C_7$ & 5523891 & 2163473 & 109 & 7687473  & 13781 \\
\hline
$D_7$ &623510 & 228140 & 44 & 851694 & 1498 \\
\hline
$E_7$ & 144937928 & 10347518 & 46 & 155285492 & weeks \\
\hline
\end{tabular}
\vspace{1mm}
\caption{Closed subsets for the ranks $3$ through $7$ irreducible root systems, up to conjugation by the respective Weyl group.  We exclude the empty set.}\label{comparison}
\end{table}
\end{footnotesize}
\end{landscape}

\section{Regular subalgebras}\label{regular}

Let $\mathfrak{g}$ be a semisimple Lie algebra  and $\mathfrak{h}$  a (fixed) Cartan subalgebra of $\mathfrak{g}$, with corresponding root system    $\Phi$. For $\alpha \in \Phi$ we denote $\mathfrak{g}_\alpha$ the
corresponding root space. Let $T\subseteq \Phi$ be a closed subsystem. Let $\mathfrak{t}$ be a subspace of $\mathfrak{h}$
containing $[\mathfrak{g}_\alpha, \mathfrak{g}_{-\alpha}]$ for $\alpha \in T$ such that $-\alpha \in T$. Then
\begin{equation}
\mathfrak{s}_{T,\mathfrak{t}} = \mathfrak{t} \oplus \bigoplus_{\alpha \in T} \mathfrak{g}_\alpha  
\end{equation}
is a regular subalgebra of $\mathfrak{g}$. Moreover, all regular subalgebras of $\mathfrak{g}$ arise in this manner.

Let $G$ be the adjoint group of $\mathfrak{g}$. There are many ways to characterize $G$. It is the identity component of the automorphism group of $\g$. It is the connected 
algebraic subgroup of $\mathrm{GL}(\mathfrak{g})$ with Lie algebra $\ad \mathfrak{g}$. Thirdly, it is the group generated by all $\exp(\ad x)$ for $x\in \mathfrak{g}_\alpha
$, $\alpha\in \Phi$.

The Weyl group of $\Phi$ naturally acts on the dual space $\mathfrak{h}^*$. Further, by identifying  $\mathfrak{h}$ and $\mathfrak{h}^*$ via the Killing form, the Weyl group
also acts on $\mathfrak{h}$. Letting $s_\alpha$ denote the reflection corresponding to $\alpha\in \Phi$, this means that
$s_\alpha(h)=h-\alpha(h)h_\alpha$, where $h_\alpha$ is the unique element of $[\mathfrak{g}_\alpha, \mathfrak{g}_{-\alpha}]$ with $\alpha(h_\alpha)=2$.


\begin{proposition}
The regular subalgebras $\mathfrak{s}_{T_1, \mathfrak{t}_1}$ and  $\mathfrak{s}_{T_2, \mathfrak{t}_2}$ are conjugate under $G$ if and only if there is a $w \in W(\Phi)$ with 
$w(T_1)=w(T_2)$ and $w( \mathfrak{t}_1)=w( \mathfrak{t}_2)$.
\end{proposition}
\begin{proof}
Let $g' \in G$ be such that $g'(\mathfrak{s}_{T_1, \mathfrak{t}_1})=\mathfrak{s}_{T_2, \mathfrak{t}_2}$. Let $\mathfrak{n}_{\mathfrak{g}} = \{ x\in \mathfrak{g} :
[x, \mathfrak{s}_{T_2, \mathfrak{t}_2}] \subset \mathfrak{s}_{T_2, \mathfrak{t}_2} \}$ be the normalizer of $\mathfrak{s}_{T_2, \mathfrak{t}_2}$. This is the Lie algebra
of the algebraic group $N_G(\mathfrak{s}_{T_2, \mathfrak{t}_2})=\{g\in G : g(\mathfrak{s}_{T_2, \mathfrak{t}_2})=\mathfrak{s}_{T_2, \mathfrak{t}_2}\}.$
Since $\mathfrak{h}$ and $g'(\mathfrak{h})$ are Cartan subalgebras of $\mathfrak{n}_{\mathfrak{g}}(\mathfrak{s}_{T_2, \mathfrak{t}_2})$, there is a $g''\in N_G(\mathfrak{s}_{T_2, \mathfrak{t}_2})$ with $g''g'(\mathfrak{h})=\mathfrak{h}$. Set $g_0=g''g'$. Note that $g_0 \in N_G(\mathfrak{h})$. There is a canonical homomorphism $N_G(\mathfrak{h})
\rightarrow W$ with kernel $Z_G(\mathfrak{h})$ and let $w\in W(\Phi)$ be the image of $g_0$. A short calculation shows that $g_0(\mathfrak{g}_\alpha)=\mathfrak{g}_{\beta}$ where
$\beta(h)= \alpha(g_0^{-1}(h))$ for $h\in \mathfrak{h}$. But this implies that $\beta =w(\alpha)$.  (This is easy to see if $w$ is a reflection, and it then follows in general.) So we see that
$w(T_1)=T_2$. Also $w(\mathfrak{t}_1)=g_0(\mathfrak{s}_{T_1, \mathfrak{t}_1} \cap \mathfrak{h}) =\mathfrak{s}_{T_2, \mathfrak{t}_2} \cap \mathfrak{h}=
\mathfrak{t}_2$.

For the reverse implication, let $g_0 \in N_G(\mathfrak{h})$ map to $w$. Then $g_0(\mathfrak{g}_\alpha)=\mathfrak{g}_{w(\alpha)}$ and
$g_0(\mathfrak{t}_1)=\mathfrak{t}_2$. It follows that $g_0(\mathfrak{s}_{T_1, \mathfrak{t}_1})=\mathfrak{s}_{T_2, \mathfrak{t}_2}$.
\end{proof}

\begin{corollary}\label{corollary}
Let $T\subseteq \Phi$ be a closed set and $W(\Phi)_T=\{w\in W(\Phi): w(T)=T  \}$ its stabilizer in $W(\Phi)$. The regular subalgebras
$\mathfrak{s}_{T, \mathfrak{t}_1}$ and  $\mathfrak{s}_{T, \mathfrak{t}_2}$ are conjugate under $G$ if and only if there is a $w\in W(\Phi)_T$ such that
$w(\mathfrak{t}_1)=\mathfrak{t}_2$.
\end{corollary}


\subsection{Regular subalgebras of the simple Lie algebra of type $A_3$}
The (complex) Lie algebra of type $A_3$ is the special linear algebra $\mathfrak{sl}(4, \mathbb{C})$ of traceless
$4\times 4$ complex matrices.
The Cartan subalgebra $\mathfrak{h}$ that we use consists of the diagonal elements of $\mathfrak{sl}(4, \mathbb{C})$. The Weyl group $W$ is isomorphic to the symmetric group $S_4$, acting on 
$\mathfrak{h}$ by permuting the diagonal entries. More precisely, for $\pi \in S_4$ we have
\begin{equation}\label{actiona3}
\pi \cdot \mathrm{diag}(a_1, a_2, a_3, a_4) =\mathrm{diag}(a_{\pi(1)}, a_{\pi(2)}, a_{\pi(3)}, a_{\pi(4)}).
\end{equation}
We have that $W$ is generated by the simple reflections $s_1, s_2, s_3$, and the isomorphism $W \rightarrow S_4$ maps
$s_1 \mapsto (1, 2)$,  $s_2 \mapsto (2, 3)$, $s_3 \mapsto (3, 4)$. A $W$-invariant scalar product on $\mathfrak{h}$
is defined by
\begin{equation}
h \cdot h'= \mathrm{trace}(h h'), ~ h, h' \in \mathfrak{h}.
\end{equation}
The simple roots of the root system of  $\mathfrak{sl}(4, \mathbb{C})$ are $\alpha_1, \alpha_2, \alpha_3$. The other positive roots are
$\alpha_1+ \alpha_2$,  $\alpha_2+ \alpha_3$, $\alpha_1+ \alpha_2+\alpha_3$. 

We now wish to classify the regular subalgebras of $\mathfrak{sl}(4, \mathbb{C})$. To do so, for each closed subset $T$ of the  root system of 
type $A_3$, we will determine the possible $\mathfrak{t} \subset \mathfrak{h}$ such that $\mathfrak{s}_{T, \mathfrak{t}}$ is a regular subalgebra.
The cases in which $\mathfrak{t}$ is $0-$ or $3-$dimensional are trivial, so we only consider the $1-$ and $2-$dimensional cases.


In the following example, we explicitly describe the process of determining the regular subalgebras obtained from a given closed subset of $A_3$. Determining the regular subalgebras for all other closed subsets of $A_3$ is obtained in a similar manner, and the complete results are summarized in Table \ref{rega31} through Table \ref{rega31b}.

\begin{example}
Consider the special closed subset of $A_3$ 
\begin{equation}
T=\{\alpha_2,  \alpha_3, \alpha_1+\alpha_2, \alpha_2+\alpha_3, \alpha_1+\alpha_2+\alpha_3\},
\end{equation}
which is located in row 2 of Table \ref{a3resultsa}. A calculation shows that the stabilizer of $T$, denoted $W_T$, is generated by 
the single element $s_1$. 

Since $T$ consists of positive root vectors, we may construct a regular subalgebra 
\begin{equation}
\mathfrak{s}_{T,\mathfrak{t}} = \mathfrak{t} \oplus \bigoplus_{\alpha \in T} \mathfrak{g}_\alpha 
\end{equation}
by letting $\mathfrak{t}$  be any subspace of the Cartan subalgebra $\mathfrak{h}$.  
Then, the possible $1$-dimensional  $\mathfrak{t}$ are spanned by the non-trivial traceless elements $h_{a_1, a_2, a_3}=\mathrm{diag}(a_1, a_2, a_3, -a_1-a_2-a_3)$ for $a_1, a_2, a_3 \in \mathbb{C}$.

Two such regular subalgebras determined by  $h_{a_1, a_2, a_3}$ and $h_{b_1, b_2, b_3}$, respectively, are $W$-conjugate if and only if $s_1 \cdot h_{a_1, a_2, a_3}$ is a scalar multiple of $h_{b_1, b_2, b_3}$ by Corollary \ref{corollary}. Considering that $s_1$ acts as the permutation $(1,2)$ on  the entries of $h_{a_1, a_2, a_3}$,  as described in Eq. \eqref{actiona3},  these regular subalgebras are $W$-conjugate if and only if   $(a_1, a_2, a_3)$ is a scalar multiple of $(b_1, b_2, b_3)$ or $(b_2, b_1, b_3)$.  

To determine the regular subalgebras constructed from $2$-dimensional $\mathfrak{t}$, we use the fact that a $2$-dimensional subspace is uniquely determined by a nonzero element
of $\mathfrak{h}$ perpendicular to it. Since the inner product is $W$-invariant, two subspaces are $W$-conjugate if and only if their perpendicular vectors are. 
Thus, the $2$-dimensional $\mathfrak{t}$ are the spaces orthogonal to $h_{a_1, a_2, a_3}$, for any $a_1, a_2, a_3 \in \mathbb{C}$, all of which are not zero.  

Using the same reasoning as in the $1$-dimensional case,  regular subalgebras with $\mathfrak{t}$ given by the space orthogonal to $h_{a_1, a_2, a_3}$, is conjugate to the space orthogonal to $h_{b_1, b_2, b_3}$ if and only if $(a_1, a_2, a_3)$ is a scalar multiple of $(b_1, b_2, b_3)$ or $(b_2, b_1, b_3)$. 

The results of this example are summarized in row 2 of Table \ref{rega31}.
\end{example}

\begin{tiny}
\begin{landscape}
\renewcommand{\arraystretch}{1.5}
\begin{table}[h!]
\begin{tabular}{|l|c|l|}
  \hline 
   \rowcolor{Gray}
  $\mathbf{T}$ & $\mathbf{W_{T}}$  & $\mathbf{\mathfrak{t}}$ \\ 
  \hline
  \hline
1&$1$& $\mathfrak{t}$ can be any subspace of $\mathfrak{h}$, and all of them yield non-conjugate regular subalgebras.\\
      \hline
      2&$s_1$ &  $1$-dimensional: Spanned by $h_{a_1, a_2, a_3}=\mathrm{diag}(a_1, a_2, a_3, -a_1-a_2-a_3)$;  $\langle h_{a_1, a_2, a_3} \rangle \sim \langle h_{b_1, b_2, b_3}\rangle$ iff $(a_1, a_2, a_3)$ is a scalar multiple of $(b_1, b_2, b_3)$
      or $(b_2, b_1, b_3)$.\\
      && $2$-dimensional: Spaces orthogonal to $h_{a_1, a_2, a_3}$ with the same conjugacy condition.\\
      \hline
      3&$s_3$&  $1$-dimensional: Spanned by $h_{a_1, a_2, a_3}=\mathrm{diag}(-a_1-a_2-a_3, a_1, a_2, a_3)$;  $\langle h_{a_1, a_2, a_3} \rangle \sim \langle h_{b_1, b_2, b_3}\rangle$ iff $(a_1, a_2, a_3)$ is a scalar multiple of $(b_1, b_2, b_3)$
      or $(b_1, b_3, b_2)$.\\
      && $2$-dimensional: Spaces orthogonal to $h_{a_1, a_2, a_3}$ with the same conjugacy condition.\\
      \hline
      4&$ s_2$&  $1$-dimensional: Spanned by $h_{a_1, a_2, a_3}=\mathrm{diag}(-a_1-a_2-a_3, a_1, a_2, a_3)$;  $\langle h_{a_1, a_2, a_3} \rangle \sim \langle h_{b_1, b_2, b_3}\rangle$ iff $(a_1, a_2, a_3)$ is a scalar multiple of $(b_1, b_2, b_3)$
      or $(b_2, b_1, b_3)$.\\
      && $2$-dimensional: Spaces orthogonal to $h_{a_1, a_2, a_3}$ with the same conjugacy condition.\\
      \hline
      5-9&$1$&$\mathfrak{t}$ can be any subspace of $\mathfrak{h}$, and all of them yield non-conjugate regular subalgebras. \\
      \hline
      10&$s_1, s_2$&$1$-dimensional: Spanned by $h_{a_1, a_2, a_3}=\mathrm{diag}(a_1, a_2, a_3, -a_1-a_2-a_3)$; $\langle h_{a_1, a_2, a_3} \rangle \sim \langle h_{b_1, b_2, b_3}\rangle$ iff
      $\{a_1, a_2, a_3\}$ is a scalar multiple of $\{b_1, b_2, b_3\}$. \\
      && Note: $W_T$ consists of every permutation of the first three coordinates. \\
       && $2$-dimensional: Spaces orthogonal to $h_{a_1, a_2, a_3}$ with the same conjugacy condition.\\
      \hline
      11&$s_2, s_3$&$1$-dimensional: Spanned by $h_{a_1, a_2, a_3}=\mathrm{diag}(-a_1-a_2-a_3, a_1, a_2, a_3)$; $\langle h_{a_1, a_2, a_3} \rangle \sim \langle h_{b_1, b_2, b_3}\rangle$ iff
      $\{a_1, a_2, a_3\}$ is a scalar multiple of $\{b_1, b_2, b_3\}$. \\
      && Note: $W_T$ consists of every permutation of the last three coordinates. \\
       && $2$-dimensional: Spaces orthogonal to $h_{a_1, a_2, a_3}$ with the same conjugacy condition.\\
      \hline
      12&$ s_1 s_3$&$1$-dimensional: Spanned by $h_{a_1, a_2, a_3}=\mathrm{diag}(a_1, a_2, a_3, -a_1-a_2-a_3)$;  $\langle h_{a_1, a_2, a_3} \rangle \sim \langle h_{b_1, b_2, b_3}\rangle$ iff they're equal or \\ &&$(a_1, a_2, a_3)$ is a scalar multiple of $(b_2, b_1, -b_1- b_2-b_3)$. \\
     &&   $2$-dimensional: Spaces orthogonal to $h_{a_1, a_2, a_3}$ with the same conjugacy condition. \\
      \hline
      13&$ s_3$&$1$-dimensional: Spanned by $h_{a_1, a_2, a_3}=\mathrm{diag}(-a_1-a_2-a_3, a_1, a_2, a_3)$;  $\langle h_{a_1, a_2, a_3} \rangle \sim \langle h_{b_1, b_2, b_3}\rangle$ iff they're equal or \\
      &&$(a_1, a_2, a_3)$ is a scalar multiple of $(b_1, b_3, b_2)$. \\
     &&   $2$-dimensional: Spaces orthogonal to $h_{a_1, a_2, a_3}$ with the same conjugacy condition.\\
      \hline
      14&$s_1s_2s_1$&$1$-dimensional: Spanned by $h_{a_1, a_2, a_3}=\mathrm{diag}(a_1, a_2, a_3, -a_1-a_2-a_3)$;  $\langle h_{a_1, a_2, a_3} \rangle \sim \langle h_{b_1, b_2, b_3}\rangle$ iff they're equal or $(a_1, a_2, a_3)$ \\ && is a scalar multiple of $(b_3, b_2, b_1)$. \\
     &&   $2$-dimensional: Spaces orthogonal to $h_{a_1, a_2, a_3}$ with the same conjugacy condition.\\
      \hline
      15&$s_1$ &  $1$-dimensional: Spanned by $h_{a_1, a_2, a_3}=\mathrm{diag}(a_1, a_2, a_3, -a_1-a_2-a_3)$;  $\langle h_{a_1, a_2, a_3} \rangle \sim \langle h_{b_1, b_2, b_3}\rangle$ iff $(a_1, a_2, a_3)$ is a scalar multiple of $(b_1, b_2, b_3)$
      or $(b_2, b_1, b_3)$.\\
      && $2$-dimensional: Spaces orthogonal to $h_{a_1, a_2, a_3}$ with the same conjugacy condition.\\
      \hline
\end{tabular}
\vspace{2mm}
\caption{Special closed subsets $T$ of $A_3$ with rows numbered according to the row numbering of Table \ref{a3resultsa};  the generators of the stabilizer of $T$,  denoted $W_T$; and the $1$- and $2$-dimensional subalgebras $\mathfrak{t}$ of the Cartan subalgebra of the regular subalgebras determined by $T$. 
}\label{rega31}
\end{table}
\end{landscape}
\end{tiny}

\begin{scriptsize}
\begin{landscape}
\renewcommand{\arraystretch}{1.5}
\begin{table}[h!]
\begin{tabular}{|l|c|l|}
  \hline 
   \rowcolor{Gray}
  $\mathbf{T}$ & $\mathbf{W_{T}}$  & $\mathbf{\mathfrak{t}}$ \\ 
  \hline
  \hline
16&$ s_1$& $1$-dimensional: $\langle \mathrm{diag}(1, -1, 0, 0)\rangle$ \\
& & $2$-dimensional: Spaces orthogonal to $\mathrm{diag}(a, a, b, -2a-b)$, and all are non-conjugate. \\
      \hline
      17&$ s_3$& $1$-dimensional:  $\langle \mathrm{diag}(0, 0, 1, -1)\rangle$\\
& & $2$-dimensional: Spaces orthogonal to $\mathrm{diag}(-a-2b, a, b, b)$, and all are non-conjugate.\\
      \hline
      18&$s_2$& $1$-dimensional:  $\langle \mathrm{diag}(0, 1, -1, 0)\rangle$\\
& & $2$-dimensional: Spaces orthogonal to $\mathrm{diag}(-a-2b, a, a, b)$, and all are non-conjugate.\\
      \hline
      19&$s_1, s_3$& $\langle \mathrm{diag}(1,-1,0,0), \mathrm{diag}(0,0,1,-1)\rangle$ \\
      \hline
      20&$s_1, s_2$& $\langle \mathrm{diag}(1,-1,0,0), \mathrm{diag}(0,1,-1,0)\rangle$ \\
      \hline
      21&$s_2, s_3$& $\langle \mathrm{diag}(0,1,-1,0), \mathrm{diag}(0,0,1,-1)\rangle$ \\
      \hline
      22&$s_1, s_3$&$1$-dimensional: $\langle \mathrm{diag}(0,0,1,-1)\rangle$ \\
      &&  $2$-dimensional: Spaces orthogonal to $h_{a, b, b}=\mathrm{diag}(-a-2b, a, b, b)$. The space orthogonal to $h_{a, b, b}$ is conjugate to that\\
      &&  orthogonal to $h_{a', b', b'}$ iff $(-a'-2b', a', b')$ is a scalar multiple of $(-a-2b, a, b)$ or $(a, -a-2b, b)$. \\
      \hline
      23&$s_1, s_3$&$1$-dimensional: $\langle \mathrm{diag}(1,-1,0, 0)\rangle$ \\

        &&  $2$-dimensional:  Spaces orthogonal to $h_{a,  a, b }=\mathrm{diag}(a, a, b, -2a-b)$. The space orthogonal to $ h_{a,  a, b }$ is conjugate to that   \\
        &&  orthogonal to $ h_{a', a', b' }$ iff $(a', b', -2a'-b')$ is a scalar multiple of $(a,  b, -2a-b)$ or $(a,  -2a-b, b )$.  \\
      \hline
      24&$s_1$& $1$-dimensional:  $\langle \mathrm{diag}(1, -1, 0, 0)\rangle$\\
      & & $2$-dimensional: Spaces orthogonal to $\mathrm{diag}(a, a, b, -2a-b)$, and all are non-conjugate. \\
      \hline
      25&$ s_2$& $1$-dimensional:  $\langle \mathrm{diag}(0, 1, -1, 0)\rangle$\\
        & & $2$-dimensional: Spaces orthogonal to $\mathrm{diag}(-2a-b, a, a, b)$, and all are non-conjugate. \\
      \hline
      26&$ s_1s_2s_1$& $1$-dimensional:  $\langle \mathrm{diag}(1, 0, -1, 0)\rangle$\\
      & & $2$-dimensional: Spaces orthogonal to $\mathrm{diag}(a, b, a, -2a-b)$, and all are non-conjugate.\\
      \hline
      27&$s_3$& $1$-dimensional:  $\langle \mathrm{diag}(0, 0, 1, -1)\rangle$\\
      & & $2$-dimensional: Spaces orthogonal to $\mathrm{diag}(-a-2b, a, b, b)$, and all are non-conjugate.\\
      \hline
      28&$s_1s_2s_1$& $1$-dimensional:  $\langle \mathrm{diag}(1, 0, -1, 0)\rangle$\\
      & & $2$-dimensional: Spaces orthogonal to $\mathrm{diag}(a, b, a, -2a-b)$, and all are non-conjugate.\\
      \hline
\end{tabular}
\vspace{2mm}
\caption{Levi decomposable closed subsets $T$ of $A_3$  with rows numbered according to the row numbering of Table \ref{a3resultsa};  the generators of the stabilizer of $T$,  denoted $W_T$; and the $1$- and $2$-dimensional subalgebras $\mathfrak{t}$ of the Cartan subalgebra of the regular subalgebras determined by $T$.}
\end{table}
\end{landscape}
\end{scriptsize}

\begin{footnotesize}
\begin{landscape}
\renewcommand{\arraystretch}{1.5}
\begin{table}[h!]
\begin{tabular}{|l|c|l|}
  \hline 
   \rowcolor{Gray}
  $\mathbf{T}$ & $\mathbf{W_T}$  & $\mathbf{\mathfrak{t}}$ \\ 
  \hline
  \hline
 $29$ & $ s_1, s_3$& $1$-dimensional: $\langle \mathrm{diag}(1, -1, 0, 0) \rangle$\\ 
  &&  $2$-dimensional:  Spaces orthogonal to $h_{a,  a, b }=\mathrm{diag}(a, a, b, -2a-b)$. The space orthogonal to $ h_{a,  a, b }$ is conjugate to that   \\
        &&  orthogonal to $ h_{a', a', b' }$ iff $(a', b', -2a'-b')$ is a scalar multiple of $(a,  b, -2a-b)$ or $(a,  -2a-b, b )$. \\
  \hline
  $30$&$s_1, s_2$ &$\langle \mathrm{diag}(1, -1, 0, 0),
 \mathrm{diag}(0, 1, -1, 0)  \rangle$\\ 
      \hline
   $31$ & $s_1, s_3, s_2s_1s_3s_2$    &$\langle \mathrm{diag}(1, -1, 0, 0),
 \mathrm{diag}(0, 0, 1, -1)  \rangle$ \\
      \hline
\end{tabular}
\vspace{2mm}
\caption{Symmetric closed subsets  $T$ of $A_3$ with rows numbered according to the row numbering of Table \ref{a3resultsb};  the generators of the stabilizer of $T$,  denoted $W_T$; and the $1$- and $2$-dimensional subalgebras $\mathfrak{t}$ of the Cartan subalgebra of the regular subalgebras determined by $T$. We omit the case $T=\Phi$, for which there are no nonempty $\mathfrak{t}$.}\label{rega31b}
\end{table}
\end{landscape}
\end{footnotesize}

\subsection{Regular subalgebras of the simple Lie algebra of type $B_3$}

The (complex) Lie algebra of type $B_3$ is the special orthogonal algebra $\mathfrak{so}(7, \mathbb{C})$. It is the Lie algebra of $7\times 7$ complex matrices satisfying
\begin{equation}
X^t M =-MX,
\end{equation}
where 
\begin{equation}
M= \begin{pmatrix}
    2& 0 &   0 \\
    0 & 0 &  I_3 \\
    0 &I_3& 0
    \end{pmatrix}.
\end{equation}

A Cartan subalgebra $\mathfrak{h}$ is given by elements of the form 
\begin{equation}
h=\mathrm{diag}(0, \lambda_1, \lambda_2, \lambda_3, -\lambda_1, -\lambda_2, -\lambda_3),
\end{equation}
 for $\lambda_1, \lambda_2, \lambda_3 \in \mathbb{C}$. 
The positive roots are 
\begin{equation}
\Phi^+ =\Bigg\{
\begin{array}{llllllllll}
\alpha_1,   &\alpha_2, &\alpha_3, \\[-2\jot]
 \alpha_1+\alpha_2, &\alpha_2+\alpha_3, &  \alpha_1+\alpha_2+ \alpha_3, \\[-2\jot]
\alpha_2+2\alpha_3, & \alpha_1+\alpha_2+2 \alpha_3, & \alpha_1+2\alpha_2+2 \alpha_3
\end{array}\Bigg\},
\end{equation} 
where
\begin{equation}
\begin{array}{lllllllll}
\alpha_1(h)&=& \lambda_1-\lambda_2,\\
\alpha_2(h)&=& \lambda_2-\lambda_3,\\
\alpha_3(h)&=& \lambda_3.
\end{array}
\end{equation}
Define a change of basis 
\begin{equation}
\beta_1=\alpha_1+\alpha_2+\alpha_3,  \quad \beta_2=\alpha_2+\alpha_3, \quad \beta_3=\alpha_3.
\end{equation}
The fundamental reflections are 
\begin{equation}
\begin{array}{llllllllllll}
s_i(\beta_i)&=&\beta_{i+1}, &s_i(\beta_{i+1})&=&\beta_i, & i=1,2, \\
s_i(\beta_j)&=&\beta_j,&&& & j\neq i, i+1, i=1,2, \\
s_3(\beta_3)&=&-\beta_3, & s_3(\beta_i)&=&\beta_i, & i\neq 3.
\end{array}
\end{equation}
The Weyl group generated by the fundamental reflections is then seen to be isomorphic to $\mathbb{Z}_2^3 \rtimes S_3$:  The isomorphism is given by $s_1 \mapsto (1,2)\in S_3$, 
$s_2 \mapsto (2, 3)\in S_3$, and $s_3 \mapsto (0,0,1)\in \mathbb{Z}_2^3$.  The actions of $s_1$, $s_2$, and $s_3$ on the simple roots $\alpha_1$, $\alpha_2$, and $\alpha_3$ are
\begin{equation}
\begin{array}{lllllllllllllllll}
s_1(\alpha_1) &=&-\alpha_1, & s_2(\alpha_1) &=&\alpha_1+\alpha_2, & s_3(\alpha_1) &=&\alpha_1, \\
s_1(\alpha_2) &=&\alpha_1+\alpha_2, & s_2(\alpha_2) &=&- \alpha_2, & s_3(\alpha_2) &=&\alpha_2+2\alpha_3, \\
s_1(\alpha_3) &=&\alpha_3 & s_2(\alpha_3) &=&\alpha_2+\alpha_3, & s_3(\alpha_3) &=& -\alpha_3.
\end{array}
\end{equation}

The correspondence, due to the Killing form, between $\mathfrak{h}^*$ and $\mathfrak{h}$  determines an action of the Weyl group on $\mathfrak{h}$.
Specifically, a computation shows that we have the correspondence
\begin{equation}
\mathrm{diag}(0, \lambda_1, \lambda_2, \lambda_3, -\lambda_1, -\lambda_2, -\lambda_3) ~ \longleftrightarrow ~ 2(\lambda_1\beta_1+\lambda_2\beta_2+\lambda_3\beta_3).
\end{equation}
The actions of $s_1, s_2$, and $s_3$ on  $\mathfrak{h}$ are given by:
\begin{small}
\begin{equation}
\begin{array}{llllll}
 s_1 \cdot \mathrm{diag}(0, \lambda_1, \lambda_2, \lambda_3, -\lambda_1, -\lambda_2, -\lambda_3) 
=\mathrm{diag}(0, \lambda_{2}, \lambda_{1}, \lambda_{3}, -\lambda_{2}, -\lambda_{1}, -\lambda_{3}),\\
 s_2 \cdot \mathrm{diag}(0, \lambda_1, \lambda_2, \lambda_3, -\lambda_1, -\lambda_2, -\lambda_3) 
=\mathrm{diag}(0, \lambda_{1}, \lambda_{3}, \lambda_{2}, -\lambda_{1}, -\lambda_{3}, -\lambda_{2}),\\
s_3 \cdot \mathrm{diag}(0, \lambda_1, \lambda_2, \lambda_3, -\lambda_1, -\lambda_2, -\lambda_3) 
=\mathrm{diag}(0, \lambda_{1}, \lambda_{2}, -\lambda_{3}, -\lambda_{1}, -\lambda_{2}, \lambda_{3}).
\end{array}
\end{equation}\end{small}

We now describe a set  of generators for $\mathfrak{so}(7, \mathbb{C})$. Number matrix rows and columns $0, 1, 2, 3, -1, -2, -3$. Then, positive root vectors $e_\alpha$ and negative root vectors $f_\alpha$ corresponding to the simple roots are 
\begin{equation}
\begin{array}{llllll}
e_{\alpha_1} &=& E_{12}-E_{-2, -1}, & f_{\alpha_1} &=& E_{21}-E_{-1,-2},\\
e_{\alpha_2} &=& E_{2 3}-E_{-3, -2}, & f_{\alpha_2} &=& E_{3 2}-E_{-2, -3},\\
e_{\alpha_3} &=& 2E_{30}-E_{0, -3}, & f_{\alpha_3} &=&  E_{0 3}-2E_{-3, 0}.
\end{array}
\end{equation}
These root vectors determine the following basis of the Cartan subalgebra $\mathfrak{h}$:
\begin{equation}
\begin{array}{lllllllll}
h_{\alpha_1}&=&  \mathrm{diag} (0,1,-1,0,-1,1,0), \\
 h_{\alpha_2}&=& \mathrm{diag} (0,0,1,-1,0,-1,1), \\
 h_{\alpha_3}&=&\mathrm{diag}  (0,0,0,2,0,0,-2).
 \end{array}
\end{equation}

A $W$-invariant inner product on $\mathfrak{h}$ is defined by
\begin{equation}
h \cdot h'= \mathrm{trace}(h h').
\end{equation}

We now wish to classify the regular subalgebras of $\mathfrak{so}(7, \mathbb{C})$. To do so, for each closed subset $T$ of the  root system of 
type $B_3$, we will determine the possible $\mathfrak{t} \subset \mathfrak{h}$ such that $\mathfrak{s}_{T, \mathfrak{t}}$ is a regular subalgebra.
The cases in which $\mathfrak{t}$ is $0-$ or $3-$dimensional are trivial, so we only consider the $1-$ and $2-$dimensional cases.

In the following example, we explicitly describe the process of determining  the regular subalgebras obtained from a given closed subset of $B_3$. Determining the regular subalgebras for all other closed subsets of $B_3$ is obtained in a similar manner, and the complete results are summarized in Tables  \ref{bsregstart} through \ref{bsregend}.

\begin{example}
Consider the Levi decomposable closed subset of $B_3$ 

\begin{equation}
T=\Bigg\{
\begin{array}{llllllllllll}
  \alpha_1, \alpha_2,  \alpha_3,  \alpha_1+\alpha_2, 
 \alpha_2+\alpha_3, 
 \alpha_1+\alpha_2+ \alpha_3,  \\  [-0.3\jot] 2\alpha_2+\alpha_3,
 \alpha_1+2\alpha_2+ \alpha_3,  2\alpha_1+2\alpha_2+ \alpha_3, -\alpha_3  
 \end{array}
 \Bigg\},
\end{equation}
which is located in row  47 of Table \ref{b3resultsa}. A calculation shows that the stabilizer of $T$, denoted $W_T$, is generated by 
the single element $s_3$.  

Since $T$ contains precisely one negative/positive root pair, $\{\alpha_3, -\alpha_3\}$, $\mathfrak{t}$ must contain $h_{\alpha_3}$. Hence, there is only one $1$-dimensional $\mathfrak{t}$, which is 
given by $\langle h_{\alpha_3} \rangle$.

We now determine the $2$-dimensional $\mathfrak{t}$, each of which must contain $h_{\alpha_3}$. Each such subspace is determined by a vector
orthogonal to it, and since $\mathfrak{t}$ contains $h_{\alpha_3}$, the orthogonal vector must be of the form $\mathrm{diag}(0,a_1, a_2,0,-a_1,-a_2,0)$. 
 Since $s_3$ fixes the vector $\mathrm{diag}(0,a_1, a_2,0,-a_1,-a_2,0)$, Corollary \ref{corollary} implies that each such subspace is non-conjugate.


The results of this example are summarized in row 47 of Table \ref{regb3ex}.
\end{example}

\begin{footnotesize}
\begin{landscape}
\renewcommand{\arraystretch}{1.5}
\begin{table}[h!]
\begin{tabular}{|l|c|l|}
  \hline 
   \rowcolor{Gray}
  $\mathbf{T}$ & $\mathbf{W_T}$  & $\mathbf{\mathfrak{t}}$ \\ 
  \hline
  \hline
1&
 $1$ & $\mathfrak{t}$ can be any subspace of $\mathfrak{h}$, and all of them yield non-conjugate
regular subalgebras.
\\
\hline
2 & $ s_3$ & $1$-dimensional:  $\langle  h_{a_1, a_2, a_3} \rangle$;   $\langle h_{a_1, a_2, a_3} \rangle \sim \langle h_{b_1, b_2, b_3} \rangle$ iff $(a_1, a_2, a_3)$ is a scalar multiple of  $(b_1, b_2, \pm b_3)$. \\
& & $2$-dimensional: Spaces
orthogonal to $h_{a_1, a_2, a_3}$, with the same conjugacy condition. \\
\hline

3 &  $s_1$ & $1$-dimensional: $\langle h_{a_1, a_2, a_3}\rangle$; 
$\langle h_{a_1, a_2, a_3} \rangle \sim \langle h_{b_1, b_2, b_3} \rangle$ iff $(a_1, a_2, a_3)$ is a scalar multiple of $(b_1, b_2, b_3)$ or $(b_2, b_1, b_3)$.\\
&&  $2$-dimensional: Spaces
orthogonal to $h_{a_1, a_2, a_3}$, with the same conjugacy condition.  \\
\hline 
4 & $s_2$ & $1$-dimensional:  $\langle h_{a_1, a_2, a_3}\rangle$;
$\langle h_{a_1, a_2, a_3} \rangle \sim \langle h_{b_1, b_2, b_3} \rangle$ is a scalar multiple of $(b_1, b_2, b_3)$ or $(b_1, b_3, b_2)$. \\
&& $2$-dimensional: Spaces
orthogonal to $h_{a_1, a_2, a_3}$, with the same conjugacy condition.  
 \\
 \hline
5 &  $s_3$ &  $1$-dimensional: $\langle h_{a_1, a_2, a_3}\rangle$;   $\langle h_{a_1, a_2, a_3} \rangle \sim \langle h_{b_1, b_2, b_3} \rangle$ iff  $(a_1, a_2, a_3)$ is a scalar multiple of  $(b_1, b_2, \pm b_3)$.  \\
&& $2$-dimensional: Spaces
orthogonal to $h_{a_1, a_2, a_3}$, with the same conjugacy condition.
\\
\hline
6, 7 &  $1$ &  $\mathfrak{t}$ can be any subspace of $\mathfrak{h}$, and all of them yield non-conjugate
regular subalgebras. \\ 
\hline
 8& $s_1, s_3$&  $1$-dimensional: $\langle h_{a_1, a_2, a_3}\rangle$;  $\langle h_{a_1, a_2, a_3} \rangle \sim \langle h_{b_1, b_2, b_3} \rangle$ iff
 $(a_1, a_2, a_3)$ is a scalar multiple of  $(a_1, a_2, \pm a_3)$ or  $(a_2, a_1, \pm a_3)$.
 \\
 && $2$-dimensional: Spaces
orthogonal to $h_{a_1, a_2, a_3}$, with the same conjugacy condition.
\\
\hline
9& $s_3$ & $1$-dimensional: $\langle h_{a_1, a_2, a_3}\rangle$;     $\langle h_{a_1, a_2, a_3} \rangle \sim \langle h_{b_1, b_2, b_3} \rangle$ iff $(a_1, a_2, a_3)$ is a scalar multiple of  $(b_1, b_2, \pm b_3)$.  \\
 && $2$-dimensional: Spaces
orthogonal to $h_{a_1, a_2, a_3}$, with the same conjugacy condition.
\\
\hline
10& $1$ & $\mathfrak{t}$ can be any subspace of $\mathfrak{h}$, and all of them yield non-conjugate
regular subalgebras.\\
\hline
11 & $s_2$ & 
    $1$-dimensional: $\langle h_{a_1, a_2, a_3}\rangle$;     $\langle h_{a_1, a_2, a_3} \rangle \sim \langle h_{b_1, b_2, b_3} \rangle$ iff $(a_1, a_2, a_3)$ is a scalar multiple of  $(b_1, b_3, b_2)$, or $(b_1, b_2, b_3)$.  \\ && $2$-dimensional: Spaces
orthogonal to $h_{a_1, a_2, a_3}$, with the same conjugacy condition. \\
\hline   
   12
     &$s_3$&     $1$-dimensional: $\langle h_{a_1, a_2, a_3}\rangle$;     $\langle h_{a_1, a_2, a_3} \rangle \sim \langle h_{b_1, b_2, b_3} \rangle$ iff $(a_1, a_2, a_3)$ is a scalar multiple of  $(b_1, b_2, \pm b_3)$.  
    \\
 && $2$-dimensional: Spaces
orthogonal to $h_{a_1, a_2, a_3}$, with the same conjugacy condition.
\\
\hline
13 & $s_3 $&  $1$-dimensional: $\langle h_{a_1, a_2, a_3}\rangle$;     $\langle h_{a_1, a_2, a_3} \rangle \sim \langle h_{b_1, b_2, b_3} \rangle$ iff $(a_1, a_2, a_3)$ is a scalar multiple of  $(b_1, b_2, \pm b_3)$.  \\
&& $2$-dimensional: Spaces
orthogonal to $h_{a_1, a_2, a_3}$, with the same conjugacy condition.
\\
\hline

14 & $1$ & $\mathfrak{t}$ can be any subspace of $\mathfrak{h}$, and all of them yield non-conjugate
regular subalgebras.
  \\
\hline
\end{tabular}
\vspace{2mm}
\caption{Special closed subsets  $T$--row numbers $1$ to $14$--of $B_3$;  the generators of the stabilizer of $T$,  denoted $W_T$; and the $1$- and $2$-dimensional subalgebras $\mathfrak{t}$ of the Cartan subalgebra of the regular subalgebras determined by $T$. In this table $h_{a_1, a_2, a_3}=\mathrm{diag}(0, a_1, a_2, a_3, -a_1, -a_2, -a_3)$.} 
\label{bsregstart}
\end{table}
\end{landscape}
\end{footnotesize}

\begin{footnotesize}
\begin{landscape}
\renewcommand{\arraystretch}{1.5}
\begin{table}[h!]
\begin{tabular}{|l|c|l|}
  \hline 
   \rowcolor{Gray}
  $\mathbf{T}$ & $\mathbf{W_T}$  & $\mathbf{\mathfrak{t}}$ \\ 
  \hline
  \hline
  15 & $s_1, s_2 $ &    $1$-dimensional: $\langle h_{a_1, a_2, a_3}\rangle$;     $\langle h_{a_1, a_2, a_3} \rangle \sim \langle h_{b_1, b_2, b_3} \rangle$ iff $(a_1, a_2, a_3)$ is a scalar multiple of  \\
  && any permutation of  the elements of $(b_1, b_2, b_3)$.  \\
&& $2$-dimensional: Spaces
orthogonal to $h_{a_1, a_2, a_3}$, with the same conjugacy condition.
\\
\hline
16 & $s_2$&    $1$-dimensional: $\langle h_{a_1, a_2, a_3}\rangle$;     $\langle h_{a_1, a_2, a_3} \rangle \sim \langle h_{b_1, b_2, b_3} \rangle$ iff  $(a_1, a_2, a_3)$ is a scalar multiple of  $(b_1, b_3, b_2)$, or $(b_1, b_2, b_3)$.  \\
&& $2$-dimensional: Spaces
orthogonal to $h_{a_1, a_2, a_3}$, with the same conjugacy condition.
\\
\hline
17, 18 & $ 1$ &  $\mathfrak{t}$ can be any subspace of $\mathfrak{h}$, and all of them yield non-conjugate
regular subalgebras. \\
\hline
19& $s_1, s_3$ &   $1$-dimensional: $\langle h_{a_1, a_2, a_3}\rangle$;     $\langle h_{a_1, a_2, a_3} \rangle \sim \langle h_{b_1, b_2, b_3} \rangle$ iff 
 $(a_1, a_2, a_3)$ is a scalar multiple of  $(a_1, a_2, \pm a_3)$ or  $(a_2, a_1, \pm a_3)$. \\
&& $2$-dimensional: Spaces
orthogonal to $h_{a_1, a_2, a_3}$, with the same conjugacy condition.
\\
\hline
20 & $ s_1s_3$ &    
$1$-dimensional: $\langle h_{a_1, a_2, a_3}\rangle$;     $\langle h_{a_1, a_2, a_3} \rangle \sim \langle h_{b_1, b_2, b_3} \rangle$ iff 
 $(a_1, a_2, a_3)$ is a scalar multiple of  $(a_1, a_2, a_3)$ or  $(a_2, a_1, -a_3)$.  \\
&& $2$-dimensional: Spaces
orthogonal to $h_{a_1, a_2, a_3}$, with the same conjugacy condition.
\\
\hline
21& $1$& $\mathfrak{t}$ can be any subspace of $\mathfrak{h}$, and all of them yield non-conjugate
regular subalgebras. \\
\hline
22& $s_3$ & $1$-dimensional: $\langle h_{a_1, a_2, a_3}\rangle$;     $\langle h_{a_1, a_2, a_3} \rangle \sim \langle h_{b_1, b_2, b_3} \rangle$ iff  $(a_1, a_2, a_3)$ is a scalar multiple of  $(b_1, b_2, \pm b_3)$. \\
&& $2$-dimensional: Spaces
orthogonal to $h_{a_1, a_2, a_3}$, with the same conjugacy condition.
\\
\hline
 23 
& $s_1s_2s_1$ &   $1$-dimensional: $\langle h_{a_1, a_2, a_3}\rangle$;     $\langle h_{a_1, a_2, a_3} \rangle \sim \langle h_{b_1, b_2, b_3} \rangle$ iff  $(a_1, a_2, a_3)$
  is a scalar multiple of  $(b_1, b_2, b_3)$, or
  $(b_3, b_2, b_1)$.  \\
&& $2$-dimensional: Spaces
orthogonal to $h_{a_1, a_2, a_3}$, with the same conjugacy condition.
\\
\hline
  
24& $s_2 s_3s_2$ &  
 $1$-dimensional: $\langle h_{a_1, a_2, a_3}\rangle$;     $\langle h_{a_1, a_2, a_3} \rangle \sim \langle h_{b_1, b_2, b_3} \rangle$ iff  $(a_1, a_2, a_3)$
  is a scalar multiple of  $(b_1, \pm b_2,  b_3)$.  \\
&& $2$-dimensional: Spaces
orthogonal to $h_{a_1, a_2, a_3}$, with the same conjugacy condition.
\\
\hline
  
25, 26 & $1$ & $\mathfrak{t}$ can be any subspace of $\mathfrak{h}$, and all of them yield non-conjugate
regular subalgebras.  \\
\hline
  
27 &  $s_3$&     $1$-dimensional: $\langle h_{a_1, a_2, a_3}\rangle$;     $\langle h_{a_1, a_2, a_3} \rangle \sim \langle h_{b_1, b_2, b_3} \rangle$ iff  $(a_1, a_2, a_3)$ is a scalar multiple of  $(b_1, b_2, \pm b_3)$.  \\
&& $2$-dimensional: Spaces
orthogonal to $h_{a_1, a_2, a_3}$, with the same conjugacy condition.
\\
\hline
  
 28 & $ 1$ & $\mathfrak{t}$ can be any subspace of $\mathfrak{h}$, and all of them yield non-conjugate
regular subalgebras. 
  \\
  \hline
\end{tabular}
\vspace{2mm}
\caption{Special closed subsets  $T$--row numbers $15$ to $28$--of $B_3$;  the generators of the stabilizer of $T$,  denoted $W_T$; and the $1$- and $2$-dimensional subalgebras $\mathfrak{t}$ of the Cartan subalgebra of the regular subalgebras determined by $T$. In this table $h_{a_1, a_2, a_3}=\mathrm{diag}(0, a_1, a_2, a_3, -a_1, -a_2, -a_3)$.}
\end{table}
\end{landscape}
\end{footnotesize}

\begin{scriptsize}
\begin{landscape}
\renewcommand{\arraystretch}{1.5}
\begin{table}[h!]
\begin{tabular}{|l|c|l|}
  \hline 
   \rowcolor{Gray}
  $\mathbf{T}$ & $\mathbf{W_T}$  & $\mathbf{\mathfrak{t}}$ \\ 
  \hline
  \hline
  29 & $s_2$ &   $1$-dimensional: $\langle h_{a_1, a_2, a_3}\rangle$;     $\langle h_{a_1, a_2, a_3} \rangle \sim \langle h_{b_1, b_2, b_3} \rangle$ iff $(a_1, a_2, a_3)$ is a scalar multiple of  $(b_1, b_2,  b_3)$,
 or $(b_1, b_3,  b_2)$.  \\
&& $2$-dimensional: Spaces
orthogonal to $h_{a_1, a_2, a_3}$, with the same conjugacy condition.
\\
\hline
30& $ s_2, s_3$&   $1$-dimensional: $\langle h_{a_1, a_2, a_3}\rangle$;     $\langle h_{a_1, a_2, a_3} \rangle \sim \langle h_{b_1, b_2, b_3} \rangle$ iff $(a_1, a_2, a_3)$
  is a scalar multiple of  $(b_1, \pm b_2, \pm b_3)$, or
  $(b_1, \pm b_3, \pm b_2)$.  \\&& $2$-dimensional: Spaces
orthogonal to $h_{a_1, a_2, a_3}$, with the same conjugacy condition.
\\
\hline

31 & $s_3$ &  
 $1$-dimensional: $\langle h_{a_1, a_2, a_3}\rangle$;     $\langle h_{a_1, a_2, a_3} \rangle \sim \langle h_{b_1, b_2, b_3} \rangle$ iff $(a_1, a_2, a_3)$ is a scalar multiple of  $(b_1, b_2, \pm b_3)$.  \\
&& $2$-dimensional: Spaces
orthogonal to $h_{a_1, a_2, a_3}$, with the same conjugacy condition.
\\
\hline
  
32 & $s_1s_3$ &  
   $1$-dimensional: $\langle h_{a_1, a_2, a_3}\rangle$;     $\langle h_{a_1, a_2, a_3} \rangle \sim \langle h_{b_1, b_2, b_3} \rangle$ iff  $(a_1, a_2, a_3)$ is a scalar multiple of  $(b_1, b_2,  b_3)$, or $(b_2, b_1,-b_3)$.\\
&& $2$-dimensional: Spaces
orthogonal to $h_{a_1, a_2, a_3}$, with the same conjugacy condition.
\\
\hline
  
33 & $s_3$ &   $1$-dimensional: $\langle h_{a_1, a_2, a_3}\rangle$;     $\langle h_{a_1, a_2, a_3} \rangle \sim \langle h_{b_1, b_2, b_3} \rangle$ iff $(a_1, a_2, a_3)$ is a scalar multiple of  $(b_1, b_2, \pm b_3)$.  \\
&& $2$-dimensional: Spaces
orthogonal to $h_{a_1, a_2, a_3}$, with the same conjugacy condition.
\\
\hline

 34 & $s_2s_3s_2, s_1s_2s_1$ & 
$1$-dimensional: $\langle h_{a_1, a_2, a_3}\rangle$;     $\langle h_{a_1, a_2, a_3} \rangle \sim \langle h_{b_1, b_2, b_3} \rangle$ iff  $(a_1, a_2, a_3)$ is a scalar multiple of
 $(b_1, \pm b_2, b_3)$, or  $(b_3, \pm b_2, b_1)$.  \\
&& $2$-dimensional: Spaces
orthogonal to $h_{a_1, a_2, a_3}$, with the same conjugacy condition.
\\
\hline
  
35& $1$ &  $\mathfrak{t}$ can be any subspace of $\mathfrak{h}$, and all of them yield non-conjugate
regular subalgebras.  \\
\hline
  
36 & $s_3, s_2s_3s_2 $& 
  $1$-dimensional: $\langle h_{a_1, a_2, a_3}\rangle$;     $\langle h_{a_1, a_2, a_3} \rangle \sim \langle h_{b_1, b_2, b_3} \rangle$ iff   $(a_1, a_2, a_3)$ is a scalar multiple of
$(b_1, \pm b_2, \pm b_3)$ (all sign combinations included). \\
&& $2$-dimensional: Spaces
orthogonal to $h_{a_1, a_2, a_3}$, with the same conjugacy condition.
\\
\hline

 37& $s_3$&     $1$-dimensional: $\langle h_{a_1, a_2, a_3}\rangle$;     $\langle h_{a_1, a_2, a_3} \rangle \sim \langle h_{b_1, b_2, b_3} \rangle$ iff  $(a_1, a_2, a_3)$ is a scalar multiple of  $(b_1, b_2, \pm b_3)$.  \\
&& $2$-dimensional: Spaces
orthogonal to $h_{a_1, a_2, a_3}$, with the same conjugacy condition.
\\
\hline
   
   38 & $s_2$ &  $1$-dimensional: $\langle h_{a_1, a_2, a_3}\rangle$;     $\langle h_{a_1, a_2, a_3} \rangle \sim \langle h_{b_1, b_2, b_3} \rangle$ iff  $(a_1, a_2, a_3)$ is a scalar multiple of  $(b_1, b_2,  b_3)$,
 or $(b_1, b_3,  b_2)$.  \\
&& $2$-dimensional: Spaces
orthogonal to $h_{a_1, a_2, a_3}$, with the same conjugacy condition.
\\
\hline

 39& $s_1, s_2$&  
  $1$-dimensional: $\langle h_{a_1, a_2, a_3}\rangle$;     $\langle h_{a_1, a_2, a_3} \rangle \sim \langle h_{b_1, b_2, b_3} \rangle$ iff  $(a_1, a_2, a_3)$ is a scalar multiple of  any permutation of the elements of $(b_1, b_3,  b_2)$.  \\
&& $2$-dimensional: Spaces
orthogonal to $h_{a_1, a_2, a_3}$, with the same conjugacy condition.
\\
\hline  
   
  40 & $s_2$ &   $1$-dimensional: $\langle h_{a_1, a_2, a_3}\rangle$;     $\langle h_{a_1, a_2, a_3} \rangle \sim \langle h_{b_1, b_2, b_3} \rangle$ iff  $(a_1, a_2, a_3)$ is a scalar multiple of  $(b_1, b_2,  b_3)$,
 or $(b_1, b_3,  b_2)$.  \\
 && $2$-dimensional: Spaces
orthogonal to $h_{a_1, a_2, a_3}$, with the same conjugacy condition.
\\
\hline

   41 &  $s_3, s_1s_2s_3s_2s_1$ &  $1$-dimensional: $\langle h_{a_1, a_2, a_3}\rangle$;     $\langle h_{a_1, a_2, a_3} \rangle \sim \langle h_{b_1, b_2, b_3} \rangle$ iff  $(a_1, a_2, a_3)$ is a scalar multiple of  $(\pm b_1, b_2, \pm b_3)$
 (all sign combinations included).  \\
   && $2$-dimensional: Spaces
orthogonal to $h_{a_1, a_2, a_3}$, with the same conjugacy condition.
\\
\hline
   
42 & $ s_2s_3s_2 $&  $1$-dimensional: $\langle h_{a_1, a_2, a_3}\rangle$;     $\langle h_{a_1, a_2, a_3} \rangle \sim \langle h_{b_1, b_2, b_3} \rangle$ iff  $(a_1, a_2, a_3)$ is a scalar multiple of  $( b_1, \pm b_2, b_3)$.  \\
   && $2$-dimensional: Spaces
orthogonal to $h_{a_1, a_2, a_3}$, with the same conjugacy condition.
\\
\hline

\end{tabular}
\vspace{2mm}
\caption{Special closed subsets  $T$--row numbers $29$ to $42$--of $B_3$;  the generators of the stabilizer of $T$,  denoted $W_T$; and the $1$- and $2$-dimensional subalgebras $\mathfrak{t}$ of the Cartan subalgebra of the regular subalgebras determined by $T$. In this table $h_{a_1, a_2, a_3}=\mathrm{diag}(0, a_1, a_2, a_3, -a_1, -a_2, -a_3)$.}
\end{table}
\end{landscape}
\end{scriptsize}

\begin{footnotesize}
\begin{landscape}
\renewcommand{\arraystretch}{1.5}
\begin{table}[h!]
\begin{tabular}{|l|c|l|}
  \hline 
   \rowcolor{Gray}
  $\mathbf{T}$ & $\mathbf{W_T}$  & $\mathbf{\mathfrak{t}}$ \\ 
  \hline
  \hline
43 & $s_2$&  $1$-dimensional: $\langle h_{a_1, a_2, a_3}\rangle$;     $\langle h_{a_1, a_2, a_3} \rangle \sim \langle h_{b_1, b_2, b_3} \rangle$ iff $(a_1, a_2, a_3)$ is a scalar multiple of  $(b_1, b_2,  b_3)$,
 or $(b_1, b_3,  b_2)$. \\ 
&& $2$-dimensional: Spaces
orthogonal to $h_{a_1, a_2, a_3}$, with the same conjugacy condition.
\\
\hline

44 & $s_2, s_3$ &       $1$-dimensional: $\langle h_{a_1, a_2, a_3}\rangle$;  $\langle h_{a_1, a_2, a_3} \rangle \sim \langle h_{b_1, b_2, b_3} \rangle$ iff  $(a_1, a_2, a_3)$ is a scalar multiple of
$(b_1, \pm b_2, \pm b_3)$, or  $(b_1, \pm b_3, \pm b_2)$\\&& (all sign combinations included). \\
&& $2$-dimensional: Spaces
orthogonal to $h_{a_1, a_2, a_3}$, with the same conjugacy condition.
\\
\hline
 
45& $s_2, s_1s_2s_3s_2s_1$&       $1$-dimensional: $\langle h_{a_1, a_2, a_3}\rangle$;     $\langle h_{a_1, a_2, a_3} \rangle \sim \langle h_{b_1, b_2, b_3} \rangle$ iff  $(a_1, a_2, a_3)$ is a scalar multiple of
$(\pm b_1,b_2, b_3)$, or  $(\pm b_1,  b_3, b_2)$. \\
&& $2$-dimensional: Spaces
orthogonal to $h_{a_1, a_2, a_3}$, with the same conjugacy condition.
\\
\hline

46& $s_3, s_2s_3s_2$&   $1$-dimensional: $\langle h_{\alpha_2}\rangle$.  
\\
&& $2$-dimensional: Spaces orthogonal to
    $\mathrm{diag}(0, a_1, a_2, a_2, -a_1, -a_2, -a_2)$;  and  the spaces spanned by  $\langle h_{a_1, a_2, a_2}\rangle \sim \langle h_{b_1, b_2, b_2}\rangle$\\
    &&   are conjugate if and only if $(a_1, a_2, a_2)$ is a scalar multiple of  $(b_1, \pm b_2, \pm b_2)$ (all sign combinations are included). \\
\hline

\end{tabular}
\vspace{2mm}
\caption{Special closed subsets  $T$--row numbers $43$ to $46$--of $B_3$;  the generators of the stabilizer of $T$,  denoted $W_T$; and the $1$- and $2$-dimensional subalgebras $\mathfrak{t}$ of the Cartan subalgebra of the regular subalgebras determined by $T$. In this table $h_{a_1, a_2, a_3}=\mathrm{diag}(0, a_1, a_2, a_3, -a_1, -a_2, -a_3)$.}
\end{table}
\end{landscape}
\end{footnotesize}

\begin{footnotesize}
\begin{landscape}
\renewcommand{\arraystretch}{1.5}
\begin{table}[h!]
\begin{tabular}{|l|c|l|}
  \hline 
   \rowcolor{Gray}
  $\mathbf{T}$ & $\mathbf{W_T}$  & $\mathbf{\mathfrak{t}}$ \\ 
  \hline
  \hline

47 & 
$s_3$&  $1$-dimensional: $\langle h_{\alpha_3}\rangle$. \\
&& $2$-dimensional: Spaces  
   orthogonal to $\mathrm{diag}(0, a_1, a_2, 0, -a_1, -a_2, 0)$, and all of them yield non-conjugate
regular subalgebras. \\
\hline

48 & $s_1$&   $1$-dimensional: $\langle h_{\alpha_1}\rangle$. \\
&& $2$-dimensional $\mathfrak{t}$ are orthogonal to
  $\mathrm{diag}(0, a_1, a_1, a_3, -a_1, -a_1, -a_3)$, and all of them yield non-conjugate
regular subalgebras. \\
\hline

49 & $s_2$ &   $1$-dimensional:  $\langle h_{\alpha_2}\rangle$. \\
&& $2$-dimensional:  Spaces orthogonal to
  $\mathrm{diag}(0, a_1, a_2, a_2, -a_1, -a_2, -a_2)$, and all of them yield non-conjugate
regular subalgebras.\\
\hline
 
50 & $s_1, s_3$&   $1$-dimensional: $\langle h_{\alpha_1}\rangle$.  \\
&& $2$-dimensional:  Spaces  orthogonal to
  $\mathrm{diag}(0, a_1, a_1, a_3, -a_1, -a_1, -a_3)$. The $2$-dimensional space determined by  \\
  && $(a_1, a_1, a_3)$ is conjugate to that determined by $(b_1, b_1, b_3)$ iff
 $(a_1, a_1, a_3)$ is a scalar multiple of  $(b_1, b_1, \pm b_3)$. \\
 \hline

51 &  $s_1, s_3$&   $1$-dimensional: $\langle h_{\alpha_3}\rangle$.  \\
&& $2$-dimensional: Spaces  orthogonal to
    $\mathrm{diag}(0, a_1, a_2, 0, -a_1, -a_2, 0)$. The $2$-dimensional space determined by \\
    && $(a_1, a_2, 0)$ is conjugate to that determined by $(b_1, b_2, 0)$  iff
 $(a_1, a_2, 0)$ is a scalar multiple of  $(b_1, b_2, 0)$ or $(b_2, b_1, 0)$. \\
 \hline

52 & $s_1, s_3$&  $1$-dimensional: None\\
 &&  $2$-dimensional: $\langle h_{\alpha_1}, h_{\alpha_3}\rangle$. \\
 \hline
  
53 &  $s_2$&   $1$-dimensional: $\langle h_{\alpha_2}\rangle$.  \\
 &&$2$-dimensional: Spaces orthogonal to
  $\mathrm{diag}(0, a_1, a_2, a_2, -a_1, -a_2, -a_2)$, and each is non-conjugate.  \\
  \hline
  
 54& $s_1, s_2$&   None\\ 
 \hline 
  
 55&  $s_1$&  $1$-dimensional: $\langle h_{\alpha_1}\rangle$. \\
 &&$2$-dimensional: Spaces orthogonal to
    $\mathrm{diag}(0, a_1, a_1, a_3, -a_1, -a_1, -a_3)$, and each is non-conjugate. 
  \\ \hline
  
56& $s_2$&   $1$-dimensional:  $\langle h_{\alpha_2}\rangle$.  \\
&&$2$-dimensional: Spaces orthogonal to
    $\mathrm{diag}(0, a_1, a_2, a_2, -a_1, -a_2, -a_2)$, and each is non-conjugate.
  
  \\ \hline
  
57 & $s_1, s_3$&   $1$-dimensional: $\langle h_{\alpha_1}\rangle$.  \\
&& $2$-dimensional: Spaces orthogonal to
    $\mathrm{diag}(0, a_1, a_1, a_3, -a_1, -a_1, -a_3)$. The $2$-dimensional space determined by \\
    && $(a_1, a_1, a_3)$ is conjugate to that determined by $(b_1, b_1, b_3)$  if and only if
 $(a_1, a_1, a_3)$ is a scalar multiple of  $(b_1, b_1, \pm b_3)$.\\ \hline

 \end{tabular}
\vspace{2mm}
\caption{Levi decomposable closed subsets  $T$--row numbers $47$ to $57$--of $B_3$;  the generators of the stabilizer of $T$,  denoted $W_T$; and the $1$- and $2$-dimensional subalgebras $\mathfrak{t}$ of the Cartan subalgebra of the regular subalgebras determined by $T$. In this table $h_{a_1, a_2, a_3}=\mathrm{diag}(0, a_1, a_2, a_3, -a_1, -a_2, -a_3)$.} \label{regb3ex}
\end{table}
\end{landscape}
\end{footnotesize}

\begin{scriptsize}
\begin{landscape}
\renewcommand{\arraystretch}{1.5}
\begin{table}[h!]
\begin{tabular}{|l|c|l|}
  \hline 
   \rowcolor{Gray}
  $\mathbf{T}$ & $\mathbf{W_T}$  & $\mathbf{\mathfrak{t}}$ \\ 
  \hline
  \hline
58 &  $s_1s_2s_1 $ &    $1$-dimensional: $\langle h_{\alpha_1+\alpha_2}\rangle$.  
\\
&&$2$-dimensional: Spaces orthogonal to
    $\mathrm{diag}(0, a_1, a_2, 2a_1, -a_1, -a_2, -2a_1)$. The $2$-dimensional space determined by \\
    && $(a_1, a_2, 2a_1)$ is conjugate to that determined by $(b_1, b_2, 2b_1)$  if and only if  $(a_1, a_2, 2a_1)$ is a scalar multiple of  $(b_1, b_2, 2b_1)$ or $(2b_1, b_2, b_1)$. \\
    \hline
  
  59
  
&$s_2, s_3$& $\langle  h_{\alpha_2}, h_{\alpha_3}\rangle$  \\
  \hline
  
60 & $s_2, s_3s_2s_3$&  $1$-dimensional: $\langle h_{\alpha_2}\rangle$  \\
&& $2$-dimensional: Spaces orthogonal to
    $\mathrm{diag}(0, a_1, a_2, a_2, -a_1, -a_2, -a_2)$. The $2$-dimensional space determined by \\ && $(a_1, a_2, a_2)$ is conjugate to that determined by $(b_1, b_2, b_2)$  if and only if  $(a_1, a_2, a_2)$ is a scalar multiple of  $(b_1, \pm b_2, \pm b_2)$\\
    && (all sign combinations are included). \\
    \hline
  
  61 & $s_3, s_2s_3s_2 $& $1$-dimensional: $\langle h_{\alpha_3}\rangle$.  \\
  && $2$-dimensional:  Spaces orthogonal to
    $\mathrm{diag}(0, a_1, a_2, 0, -a_1, -a_2, 0)$. The $2$-dimensional space determined by \\&& $(a_1, a_2, 0)$ is conjugate to that determined by $(b_1, b_2, 0)$  if and only if  $(a_1, a_2, 0)$ is a scalar multiple of  $(b_1, \pm b_2, 0)$.\\ \hline
  
62 &  $ s_2, s_3$&   $\langle  h_{\alpha_2}, h_{\alpha_3}\rangle$.  \\
  \hline

63 & $s_2 $&    $1$-dimensional: $\langle h_{\alpha_2}\rangle$.   \\
&&$2$-dimensional: Spaces orthogonal to
    $\mathrm{diag}(0, a_1, a_2, a_2, -a_1, -a_2, -a_2)$, and each is non-conjugate.
\\
\hline
  
64 & $s_2, s_3s_2s_3$ &    $1$-dimensional: $\langle h_{\alpha_2}\rangle$.  \\
&&$2$-dimensional: Spaces orthogonal to
    $\mathrm{diag}(0, a_1, a_2, a_2, -a_1, -a_2, -a_2)$ and  the spaces determined by \\
    &&  $h_{a_1, a_2, a_2}$ and  $h_{b_1, b_2, b_2}$ are conjugate if and only if $(a_1, a_2, a_2)$
   is a scalar multiple of  $(b_1, \pm b_2, \pm b_2)$ (all sign combinations are included).
  \\
  \hline
65 & $s_2, s_3$ &   $\langle  h_{\alpha_2}, h_{\alpha_3}\rangle$ \\
\hline
  
66 &  $s_3$&   $1$-dimensional: $\langle h_{\alpha_3}\rangle$.  \\
&& $2$-dimensional: Spaces  
   orthogonal to $\mathrm{diag}(0, a_1, a_2, 0, -a_1, -a_2, 0)$, and all of them yield non-conjugate
regular subalgebras.\\
\hline

  67& $s_2s_3s_2, s_1s_2s_1$& $1$-dimensional: $\langle h_{\alpha_1+\alpha_2}\rangle$. \\&& $2$-dimensional: Spaces     orthogonal to $\mathrm{diag}(0, a_1, a_2, 2a_1, -a_1, -a_2, -2a_1)$. The $2$-dimensional space determined by \\
  && $(a_1, a_2, 2a_1)$ is conjugate to that determined by $(b_1, b_2, 2b_1)$  if and only if  $(a_1, a_2, 2a_1)$ is a scalar multiple of  $(b_1, \pm b_2, 2b_1)$ or $(2b_1, \pm b_2, b_1)$.
  \\
  \hline
 
68 &  $s_3, s_2s_3s_2$&   
$1$-dimensional: $\langle h_{\alpha_3}\rangle$.\\&& $2$-dimensional: Spaces 
   orthogonal to $\mathrm{diag}(0, a_1, a_2, 0, -a_1, -a_2, 0)$. The $2$-dimensional space determined by \\
   &&$(a_1, a_2, 0)$ is conjugate to that determined by $(b_1, b_2, 0)$  if and only if  $(a_1, a_2, 0)$ is a scalar multiple of  $(b_1, \pm b_2, 0)$.
   \\
   \hline

 \end{tabular}
\vspace{2mm}
\caption{Levi decomposable closed subsets  $T$--row numbers $58$ to $68$--of $B_3$;  the generators of the stabilizer of $T$,  denoted $W_T$; and the $1$- and $2$-dimensional subalgebras $\mathfrak{t}$ of the Cartan subalgebra of the regular subalgebras determined by $T$. In this table $h_{a_1, a_2, a_3}=\mathrm{diag}(0, a_1, a_2, a_3, -a_1, -a_2, -a_3)$.}
\end{table}
\end{landscape}
\end{scriptsize}

\begin{scriptsize}
\begin{landscape}
\renewcommand{\arraystretch}{1.5}
\begin{table}[h!]
\begin{tabular}{|l|c|l|}
  \hline 
   \rowcolor{Gray}
  $\mathbf{T}$ & $\mathbf{W_T}$  & $\mathbf{\mathfrak{t}}$ \\ 
  \hline
  \hline
  69& $s_2$&  $1$-dimensional: $\langle h_{\alpha_2}\rangle$.    \\
  &&$2$-dimensional: Spaces orthogonal to
    $\mathrm{diag}(0, a_1, a_2, a_2, -a_1, -a_2, -a_2)$, and each is non-conjugate.
\\
\hline

70& $s_1, s_2 $ &     
    None. (We must have $\mathfrak{t}=\mathfrak{h}$).\\
    \hline

71 &  $s_1$& $1$-dimensional: $\langle h_{\alpha_1}\rangle$.  \\
&& $2$-dimensional: Spaces are orthogonal to
    $\mathrm{diag}(0, a_1, a_1, a_3, -a_1, -a_1, -a_3)$, and each is non-conjugate.  \\
    \hline
   
72 & $s_2$& $1$-dimensional: $\langle h_{\alpha_2}\rangle$.    \\ 
&&$2$-dimensional: Spaces orthogonal to
    $\mathrm{diag}(0, a_1, a_2, a_2, -a_1, -a_2, -a_2)$, and each is non-conjugate. \\
    \hline
   
73& $s_3, s_1s_2s_3s_2s_1$& 
    $1$-dimensional: $\langle h_{\alpha_1+\alpha_2+\alpha_3}\rangle$.   \\&&$2$-dimensional: Spaces orthogonal to
    $\mathrm{diag}(0,0,a_2,a_3, 0, -a_2, -a_3)$ and the space determined by \\&& $(0, a_2, a_3)$ is conjugate to that determined by $(0, b_2, b_3)$
    if and only if $(0, a_2, a_3)$ is a scalar multiple of $(0, b_2, \pm b_3)$. \\ \hline

 74&  $s_2$&  $1$-dimensional: $\langle h_{\alpha_2}\rangle$.    \\
 && $2$-dimensional: Spaces orthogonal to
    $\mathrm{diag}(0, a_1, a_2, a_2, -a_1, -a_2, -a_2)$, and each is non-conjugate. \\
    \hline

   75 & $s_3s_2s_3, s_2s_3s_2s_3$& 
 $1$-dimensional: $\langle h_{\alpha_2}\rangle$.  \\&&
 $2$-dimensional: Spaces orthogonal to
    $\mathrm{diag}(0, a_1, a_2, a_2, -a_1, -a_2, -a_2)$   and the space determined by \\&& $(a_1, a_2, a_2)$ is conjugate to that determined by $(b_1, b_2, b_2)$
    if and only if $(a_1, a_2, a_2)$ is a scalar multiple of $(b_1, \pm b_2, \pm b_2)$ (the signs match).\\
    \hline

 76 & $s_2, s_3 $&   $\langle h_{\alpha_2}, h_{\alpha_3}\rangle$ \\
 \hline
   
  77 & $Ws_2, s_1s_2s_3s_2s_1$&  $1$-dimensional: $\langle h_{\alpha_2}\rangle$.  \\
  && $2$-dimensional: Spaces orthogonal to
    $\mathrm{diag}(0, a_1, a_2, a_2, -a_1, -a_2, -a_2)$   and the space determined by \\&& $(a_1, a_2, a_2)$ is conjugate to that determined by $(b_1, b_2, b_2)$
    if and only if $(a_1, a_2, a_2)$ is a scalar multiple of $(\pm b_1,  b_2,  b_2)$.
   \\ \hline
78& $s_3, s_1s_2s_3s_2s_1$& 
     $1$-dimensional: $\langle h_{\alpha_1+\alpha_2+\alpha_3}\rangle$.   
     \\
     && $2$-dimensional: Spaces orthogonal to
    $\mathrm{diag}(0,0,a_2,a_3, 0, -a_2, -a_3)$ and the space determined by \\&& $(0, a_2, a_3)$ is conjugate to that determined by $(0, b_2, b_3)$
    if and only if $(0, a_2, a_3)$ is a scalar multiple of $(0, b_2, \pm b_3)$ \\ \hline

  
    79 &  $s_3, s_1s_2s_3s_2s_1$& $\langle h_{\alpha_2},2 h_{\alpha_1} +h_{\alpha_3}\rangle$

   \\
   \hline

 \end{tabular}
\vspace{2mm}
\caption{Levi decomposable closed subsets  $T$--row numbers $69$ to $79$--of $B_3$;  the generators of the stabilizer of $T$,  denoted $W_T$; and the $1$- and $2$-dimensional subalgebras $\mathfrak{t}$ of the Cartan subalgebra of the regular subalgebras determined by $T$. In this table $h_{a_1, a_2, a_3}=\mathrm{diag}(0, a_1, a_2, a_3, -a_1, -a_2, -a_3)$.}
\label{bsregend}
\end{table}
\end{landscape}
\end{scriptsize}

\begin{footnotesize}
\begin{landscape}
\renewcommand{\arraystretch}{1.5}
\begin{table}[h!]
\begin{tabular}{|l|c|l|}
  \hline 
   \rowcolor{Gray}
  $\mathbf{T}$ & $\mathbf{W_T}$  & $\mathbf{\mathfrak{t}}$ \\ 
  \hline   \hline
  80 & $s_1, s_3, s_2s_1s_3s_2s_1s_3s_2$ & $1$-dimensional:  $\langle h_{\alpha_1}\rangle$ \\
  && $2$-dimensional: Spaces orthogonal to $h_{a, a, b}$. The space orthogonal to $h_{a, a, b}$ is conjugate to that \\
  && orthogonal to $h_{a', a', b'}$ iff $(a',a',b')$ is a scalar multiple of $(a,a,\pm b)$.\\
  \hline
  81  &$s_1, s_3, s_2s_3s_2$& $1$-dimensional:  $\langle  h_{\alpha_3} \rangle$  \\
  && $2$-dimensional: Spaces ortogonal to $h_{a,b,0}$. The space orthogonal to $h_{a,b,0}$ is conjugate to that\\
  && orthogonal to $h_{a',b',0}$ iff $(a',b')$ is a scalar multiple of $(a, \pm b)$ or $(b, \pm a)$.\\
  \hline
\end{tabular}
\vspace{2mm}
\caption{Symmetric closed subsets  $T$ of $B_3$;  the generators of the stabilizer of $T$,  denoted $W_T$; and the $1$- and $2$-dimensional subalgebras $\mathfrak{t}$ of the Cartan subalgebra of the regular subalgebras determined by $T$. We include only those closed systems of rank $1$, as the others cases are straightforward. In this table $h_{a_1, a_2, a_3}=\mathrm{diag}(0, a_1, a_2, a_3, -a_1, -a_2, -a_3)$.}
\end{table}
\end{landscape}
\end{footnotesize}

\subsection{Regular subalgebras of the simple Lie algebra of type $C_3$}

The (complex) Lie algebra of type $C_3$ is the symplectic algebra $\mathfrak{sp}(6, \mathbb{C})$. 
It is the Lie algebra of $6\times 6$ complex matrices satisfying
\begin{equation}
X^t M =-MX,
\end{equation}
where 
\begin{equation}
M= \begin{pmatrix}
     0 &  I_3 \\
    -I_3& 0
    \end{pmatrix}.
\end{equation}

A Cartan subalgebra $\mathfrak{h}$ is given by elements of the form 
\begin{equation}
h=\mathrm{diag}(\lambda_1, \lambda_2, \lambda_3, -\lambda_1, -\lambda_2, -\lambda_3),
\end{equation}
 for $\lambda_1, \lambda_2, \lambda_3 \in \mathbb{C}$. 
The positive roots are 
\begin{equation}
\Phi^+ =\Bigg\{
\begin{array}{llllllllll}
\alpha_1,   &\alpha_2, &\alpha_3, \\[-2\jot]
 \alpha_1+\alpha_2, &\alpha_2+\alpha_3, &  \alpha_1+\alpha_2+ \alpha_3, \\[-2\jot]
2\alpha_2+\alpha_3, & \alpha_1+2\alpha_2+ \alpha_3, & 2\alpha_1+2\alpha_2+ \alpha_3
\end{array}\Bigg\},
\end{equation} 
where
\begin{equation}
\begin{array}{lllllllll}
\alpha_1(h)&=& \lambda_1-\lambda_2,\\
\alpha_2(h)&=& \lambda_2-\lambda_3,\\
\alpha_3(h)&=& 2\lambda_3.
\end{array}
\end{equation}
Define a change of basis 
\begin{equation}
\beta_1=\alpha_1+\alpha_2+\frac{1}{2}\alpha_3,  \quad \beta_2=\alpha_2+\frac{1}{2}\alpha_3, \quad \beta_3=\frac{1}{2}\alpha_3.
\end{equation}
The fundamental reflections are the same as those for $B_3$:
\begin{equation}
\begin{array}{llllllllllll}
s_i(\beta_i)&=&\beta_{i+1}, &s_i(\beta_{i+1})&=&\beta_i, & i=1,2, \\
s_i(\beta_j)&=&\beta_j,&&& & j\neq i, i+1, i=1,2, \\
s_3(\beta_3)&=&-\beta_3, & s_3(\beta_i)&=&\beta_i, & i\neq 3.
\end{array}
\end{equation}
The Weyl group generated by the fundamental reflections is then seen to be isomorphic to $\mathbb{Z}_2^3 \rtimes S_3$:  The isomorphism is given by $s_1 \mapsto (1,2)\in S_3$, 
$s_2 \mapsto (2, 3)\in S_3$, and $s_3 \mapsto (0,0,1)\in \mathbb{Z}_2^3$.  The actions of $s_1$, $s_2$, and $s_3$ on the simple roots $\alpha_1$, $\alpha_2$, and $\alpha_3$ are
\begin{equation}
\begin{array}{lllllllllllllllll}
s_1(\alpha_1) &=&-\alpha_1, & s_2(\alpha_1) &=&\alpha_1+\alpha_2, & s_3(\alpha_1) &=&\alpha_1, \\
s_1(\alpha_2) &=&\alpha_1+\alpha_2, & s_2(\alpha_2) &=&- \alpha_2, & s_3(\alpha_2) &=&\alpha_2+\alpha_3, \\
s_1(\alpha_3) &=&\alpha_3 & s_2(\alpha_3) &=&2\alpha_2+\alpha_3, & s_3(\alpha_3) &=& -\alpha_3.
\end{array}
\end{equation}

The correspondence, due to the Killing form, between $\mathfrak{h}^*$ and $\mathfrak{h}$  determines an action of the Weyl group on $\mathfrak{h}$.
Specifically, a computation shows that we have the correspondence
\begin{equation}
\mathrm{diag}(\lambda_1, \lambda_2, \lambda_3, -\lambda_1, -\lambda_2, -\lambda_3) ~ \longleftrightarrow ~ 2(\lambda_1\beta_1+\lambda_2\beta_2+\lambda_3\beta_3).
\end{equation}
The actions of $s_1, s_2$, and $s_3$ on  $\mathfrak{h}$ are given by:
\begin{small}
\begin{equation}
\begin{array}{llllll}
 s_1 \cdot \mathrm{diag}( \lambda_1, \lambda_2, \lambda_3, -\lambda_1, -\lambda_2, -\lambda_3) 
=\mathrm{diag}(\lambda_{2}, \lambda_{1}, \lambda_{3}, -\lambda_{2}, -\lambda_{1}, -\lambda_{3}),\\
 s_2 \cdot \mathrm{diag}(\lambda_1, \lambda_2, \lambda_3, -\lambda_1, -\lambda_2, -\lambda_3) 
=\mathrm{diag}( \lambda_{1}, \lambda_{3}, \lambda_{2}, -\lambda_{1}, -\lambda_{3}, -\lambda_{2}),\\
s_3 \cdot \mathrm{diag}( \lambda_1, \lambda_2, \lambda_3, -\lambda_1, -\lambda_2, -\lambda_3) 
=\mathrm{diag}( \lambda_{1}, \lambda_{2}, -\lambda_{3}, -\lambda_{1}, -\lambda_{2}, \lambda_{3}).
\end{array}
\end{equation}\end{small}

We now describe a set  of generators for $\mathfrak{sp}(6, \mathbb{C})$. Number matrix rows and columns $1, 2, 3, -1, -2, -3$. Then, positive root vectors $e_\alpha$ and negative root vectors $f_\alpha$ corresponding to the simple roots are 
\begin{equation}
\begin{array}{llllll}
e_{\alpha_1} &=& E_{12}-E_{-2, -1}, & f_{\alpha_1} &=& E_{21}-E_{-1,-2},\\
e_{\alpha_2} &=& E_{2 3}-E_{-3, -2}, & f_{\alpha_2} &=& E_{32}-E_{-2, -3},\\
e_{\alpha_3} &=& E_{3,-3}, & f_{\alpha_3} &=&  E_{-3, 3}.
\end{array}
\end{equation}
These root vectors determine the following basis of the Cartan subalgebra $\mathfrak{h}$:
\begin{equation}
\begin{array}{lllllllll}
h_{\alpha_1}&=&  \mathrm{diag} (1,-1,0,-1,1,0), \\
 h_{\alpha_2}&=& \mathrm{diag} (0,1,-1,0,-1,1), \\
 h_{\alpha_3}&=&\mathrm{diag}  (0,0,1,0,0,-1).
 \end{array}
\end{equation}

A $W$-invariant inner product on $\mathfrak{h}$ is defined by
\begin{equation}
h \cdot h'= \mathrm{trace}(h h').
\end{equation}

We now wish to classify the regular subalgebras of $\mathfrak{sp}(6, \mathbb{C})$. To do so, for each closed subset $T$ of the  root system of 
type $C_3$, we will determine the possible $\mathfrak{t} \subset \mathfrak{h}$ such that $\mathfrak{s}_{T, \mathfrak{t}}$ is a regular subalgebra.
The cases in which $\mathfrak{t}$ is $0-$ or $3-$dimensional are trivial, so we only consider the $1-$ and $2-$dimensional cases.

In the following example, we explicitly describe the process of determining the regular subalgebras obtained from a given closed subset of $C_3$. Determining the regular subalgebras for all other closed subsets of $C_3$ is obtained in a similar manner, and the complete results are summarized in Tables  \ref{csregstart} through \ref{csregend}.

\begin{example}
Consider the symmetric closed subset $T=\{\alpha_1, -\alpha_1 \}$ of $C_3$,
which is located in row  79 of Table \ref{b3resultsb}. A calculation shows that the stabilizer of $T$, denoted $W_T$, is generated by 
the  elements $s_1, s_3$, and $s_2s_1s_3s_2s_1s_3s_2$.  

Clearly there is only one $1$-dimensional $\mathfrak{t}$, which is 
given by $\langle h_{\alpha_1} \rangle$.
We now determine the $2$-dimensional $\mathfrak{t}$, each of which must contain $h_{\alpha_1}$. Each such subspace is determined by a vector
orthogonal to it, and since $\mathfrak{t}$ contains $h_{\alpha_1}$, the orthogonal vector must be of the form $\mathrm{diag}(a, a, b,-a, -a, -b)$. 

 The actions of the generators of $W_T$ on  $\mathrm{diag}(a, a, b, -a, -a, -b)$ are
 \begin{equation}
 \begin{array}{rllllll}
 s_1 \cdot \mathrm{diag}(a, a, b, -a, -a, -b) &=& \mathrm{diag}(a, a, b, -a, -a, -b),\\
  s_3 \cdot \mathrm{diag}(a, a, b, -a, -a, -b) &=& \mathrm{diag}(a, a, -b, -a, -a, b),\\
   s_2s_1s_3s_2s_1s_3s_2 \cdot \mathrm{diag}(a, a, b, -a, -a, -b) &=& \mathrm{diag}(-a, -a, b, a, a, -b).
 \end{array}
 \end{equation}
 Hence, by Corollary \ref{corollary}, the space orthogonal to $\mathrm{diag}(0,a, a, b, -a, -a, -b)$ is conjugate to that 
  orthogonal to $\mathrm{diag}(a', a', b', -a', -a', -b')$ if and only if  $(a',a',b')$ is a scalar multiple of $(a,a,\pm b)$.


The results of this example are summarized in row 79 of Table \ref{csregend}.
\end{example}

\begin{small}
\begin{landscape}
\renewcommand{\arraystretch}{1.5}
\begin{table}[h!]
\begin{tabular}{|l|c|l|}
  \hline 
   \rowcolor{Gray}
  $\mathbf{T}$ & $\mathbf{W_T}$  & $\mathbf{\mathfrak{t}}$ \\ 
  \hline
1&  $1$ & $\mathfrak{t}$ can be any subspace of $\mathfrak{h}$, and all of them yield non-conjugate
regular subalgebras. \\
\hline

2& $s_1$&   $1$-dimensional:  $\langle h_{a_1, a_2, a_3}\rangle$; 
$\langle h_{a_1, a_2, a_3}\rangle \sim \langle h_{b_1, b_2, b_3}\rangle$  iff $(a_1, a_2, a_3)$ is a scalar multiple of  $(b_1, b_2,  b_3)$, or $(b_2, b_1,  b_3)$.  \\
&& $2$-dimensional: Spaces
orthogonal to $h_{a_1, a_2, a_3}$, with the same conjugacy condition. \\
\hline

3& $s_3$& $1$-dimensional:  $\langle h_{a_1, a_2, a_3}\rangle$; 
$\langle h_{a_1, a_2, a_3}\rangle \sim \langle h_{b_1, b_2, b_3}\rangle$  iff $(a_1, a_2, a_3)$ is a scalar multiple of  $(b_1, b_2, \pm b_3)$.  \\
&& $2$-dimensional: Spaces
orthogonal to $h_{a_1, a_2, a_3}$, with the same conjugacy condition. \\
\hline

4& $s_2$ &  $1$-dimensional:  $\langle h_{a_1, a_2, a_3}\rangle$; 
$\langle h_{a_1, a_2, a_3}\rangle \sim \langle h_{b_1, b_2, b_3}\rangle$  iff $(a_1, a_2, a_3)$ is a scalar multiple of  $(b_1, b_2,  b_3)$, or
 $(b_1, b_3,  b_2)$.  \\
 && $2$-dimensional: Spaces
orthogonal to $h_{a_1, a_2, a_3}$, with the same conjugacy condition. \\
\hline

5, 6 & $1$& $\mathfrak{t}$ can be any subspace of $\mathfrak{h}$, and all of them yield non-conjugate
regular subalgebras. \\ \hline

7& $s_2$&   $1$-dimensional:  $\langle h_{a_1, a_2, a_3}\rangle$; 
$\langle h_{a_1, a_2, a_3}\rangle \sim \langle h_{b_1, b_2, b_3}\rangle$  iff $(a_1, a_2, a_3)$ is a scalar multiple of  $(b_1, b_2,  b_3)$, or
 $(b_1, b_3,  b_2)$.  \\
 && $2$-dimensional: Spaces
orthogonal to $h_{a_1, a_2, a_3}$, with the same conjugacy condition. \\
\hline

8& $s_1, s_3$&   $1$-dimensional:  $\langle h_{a_1, a_2, a_3}\rangle$; 
$\langle h_{a_1, a_2, a_3}\rangle \sim \langle h_{b_1, b_2, b_3}\rangle$  iff  $(a_1, a_2, a_3)$ is a scalar multiple of  $(b_1, b_2,  \pm b_3)$, or
 $(b_2, b_1,  \pm b_3)$.  \\
 && $2$-dimensional: Spaces
orthogonal to $h_{a_1, a_2, a_3}$, with the same conjugacy condition. \\
\hline

9, 10 & $1$& $\mathfrak{t}$ can be any subspace of $\mathfrak{h}$, and all of them yield non-conjugate
regular subalgebras. \\ \hline

11 & $s_3$& $1$-dimensional:  $\langle h_{a_1, a_2, a_3}\rangle$; 
$\langle h_{a_1, a_2, a_3}\rangle \sim \langle h_{b_1, b_2, b_3}\rangle$  iff $(a_1, a_2, a_3)$ is a scalar multiple of  $(b_1, b_2, \pm b_3)$.  \\
 && $2$-dimensional: Spaces
orthogonal to $h_{a_1, a_2, a_3}$, with the same conjugacy condition. \\
\hline

 \end{tabular}
\vspace{2mm}
\caption{Special closed subsets  $T$--row numbers $1$ to $11$--of $C_3$;  the generators of the stabilizer of $T$,  denoted $W_T$; and the $1$- and $2$-dimensional subalgebras $\mathfrak{t}$ of the Cartan subalgebra of the regular subalgebras determined by $T$. In this table $h_{a_1, a_2, a_3}=\mathrm{diag}(a_1, a_2, a_3, -a_1, -a_2, -a_3)$.}
\label{csregstart}
\end{table}
\end{landscape}
\end{small}

\begin{footnotesize}
\begin{landscape}
\renewcommand{\arraystretch}{1.5}
\begin{table}[h!]
\begin{tabular}{|l|c|l|}
  \hline 
   \rowcolor{Gray}
  $\mathbf{T}$ & $\mathbf{W_T}$  & $\mathbf{\mathfrak{t}}$ \\ 
  \hline
12 & $s_2$& $1$-dimensional:  $\langle h_{a_1, a_2, a_3}\rangle$; 
$\langle h_{a_1, a_2, a_3}\rangle \sim \langle h_{b_1, b_2, b_3}\rangle$  iff $(a_1, a_2, a_3)$ is a scalar multiple of  $(b_1, b_2,  b_3)$, or
 $(b_1, b_3,  b_2)$.  \\
&& $2$-dimensional: Spaces
orthogonal to $h_{a_1, a_2, a_3}$, with the same conjugacy condition. \\
\hline

13& $1$& $\mathfrak{t}$ can be any subspace of $\mathfrak{h}$, and all of them yield non-conjugate
regular subalgebras.\\
\hline

14&  $s_1, s_2$&
$1$-dimensional:  $\langle h_{a_1, a_2, a_3}\rangle$; 
$\langle h_{a_1, a_2, a_3}\rangle \sim \langle h_{b_1, b_2, b_3}\rangle$  iff $(a_1, a_2, a_3)$ is a scalar multiple   
of \\ && any permutation of the elements of $(b_1, b_2, b_3)$.  \\
&& $2$-dimensional: Spaces
orthogonal to $h_{a_1, a_2, a_3}$, with the same conjugacy condition. \\
\hline

15, 16, 17 & $1$&  $\mathfrak{t}$ can be any subspace of $\mathfrak{h}$, and all of them yield non-conjugate
regular subalgebras. \\ \hline

18 & $s_1s_3$& 
$1$-dimensional:  $\langle h_{a_1, a_2, a_3}\rangle$; 
$\langle h_{a_1, a_2, a_3}\rangle \sim \langle h_{b_1, b_2, b_3}\rangle$  iff $(a_1, a_2, a_3)$ is a scalar multiple   
of $(b_1, b_2, b_3)$, or $(b_2, b_1, -b_3)$. \\
&& $2$-dimensional: Spaces
orthogonal to $h_{a_1, a_2, a_3}$, with the same conjugacy condition. \\
\hline

19&  $s_2$&  $1$-dimensional:  $\langle h_{a_1, a_2, a_3}\rangle$; 
$\langle h_{a_1, a_2, a_3}\rangle \sim \langle h_{b_1, b_2, b_3}\rangle$  iff $(a_1, a_2, a_3)$ is a scalar multiple of  $(b_1, b_2,  b_3)$, or
 $(b_1, b_3,  b_2)$.  \\
&& $2$-dimensional: Spaces
orthogonal to $h_{a_1, a_2, a_3}$, with the same conjugacy condition. \\
\hline

20 & $s_3$&  $1$-dimensional:  $\langle h_{a_1, a_2, a_3}\rangle$; 
$\langle h_{a_1, a_2, a_3}\rangle \sim \langle h_{b_1, b_2, b_3}\rangle$  iff $(a_1, a_2, a_3)$ is a scalar multiple of  $(b_1, b_2, \pm b_3)$.  
\\
&& $2$-dimensional: Spaces
orthogonal to $h_{a_1, a_2, a_3}$, with the same conjugacy condition. \\
\hline

21& $ s_1s_2s_1$&    $1$-dimensional:  $\langle h_{a_1, a_2, a_3}\rangle$; 
$\langle h_{a_1, a_2, a_3}\rangle \sim \langle h_{b_1, b_2, b_3}\rangle$  iff $(a_1, a_2, a_3)$ is a scalar multiple of  $(b_1, b_2, b_3)$, or $(b_3, b_2, b_1)$.  \\
&& $2$-dimensional: Spaces
orthogonal to $h_{a_1, a_2, a_3}$, with the same conjugacy condition. \\
\hline

22& $s_2, s_3$&  $1$-dimensional:  $\langle h_{a_1, a_2, a_3}\rangle$; 
$\langle h_{a_1, a_2, a_3}\rangle \sim \langle h_{b_1, b_2, b_3}\rangle$  iff  $(a_1, a_2, a_3)$ is a scalar multiple of  $(b_1, \pm b_2, \pm b_3)$, or $(b_1, \pm b_3, \pm b_2)$ \\
&& (all sign combinations included).  \\
&& $2$-dimensional: Spaces
orthogonal to $h_{a_1, a_2, a_3}$, with the same conjugacy condition. \\
\hline

 \end{tabular}
\vspace{2mm}
\caption{Special closed subsets  $T$--row numbers $12$ to $22$--of $C_3$;  the generators of the stabilizer of $T$,  denoted $W_T$; and the $1$- and $2$-dimensional subalgebras $\mathfrak{t}$ of the Cartan subalgebra of the regular subalgebras determined by $T$. In this table $h_{a_1, a_2, a_3}=\mathrm{diag}(a_1, a_2, a_3, -a_1, -a_2, -a_3)$.}
\end{table}
\end{landscape}
\end{footnotesize}

\begin{footnotesize}
\begin{landscape}
\renewcommand{\arraystretch}{1.5}
\begin{table}[h!]
\begin{tabular}{|l|c|l|}
  \hline 
   \rowcolor{Gray}
  $\mathbf{T}$ & $\mathbf{W_T}$  & $\mathbf{\mathfrak{t}}$ \\ 
  \hline
23& $s_2s_3s_2$& $1$-dimensional:  $\langle h_{a_1, a_2, a_3}\rangle$; 
$\langle h_{a_1, a_2, a_3}\rangle \sim \langle h_{b_1, b_2, b_3}\rangle$  iff $(a_1, a_2, a_3)$ is a scalar multiple of  $(b_1, \pm b_2,  b_3)$.  \\
&& $2$-dimensional: Spaces
orthogonal to $h_{a_1, a_2, a_3}$, with the same conjugacy condition. \\
\hline

24 & $1$& $\mathfrak{t}$ can be any subspace of $\mathfrak{h}$, and all of them yield non-conjugate
regular subalgebras. \\ \hline

25 & $s_3$& $1$-dimensional:  $\langle h_{a_1, a_2, a_3}\rangle$; 
$\langle h_{a_1, a_2, a_3}\rangle \sim \langle h_{b_1, b_2, b_3}\rangle$  iff $(a_1, a_2, a_3)$ is a scalar multiple of  $(b_1, b_2, \pm b_3)$.  \\
&& $2$-dimensional: Spaces
orthogonal to $h_{a_1, a_2, a_3}$, with the same conjugacy condition. \\
\hline

26 & $s_1s_2s_1$&  $1$-dimensional:  $\langle h_{a_1, a_2, a_3}\rangle$; 
$\langle h_{a_1, a_2, a_3}\rangle \sim \langle h_{b_1, b_2, b_3}\rangle$  iff  $(a_1, a_2, a_3)$ is a scalar multiple of  $(b_1, b_2, b_3)$, or $(b_3, b_2, b_1)$.  \\&& $2$-dimensional: Spaces
orthogonal to $h_{a_1, a_2, a_3}$, with the same conjugacy condition. \\
\hline

27 & $s_2$& 
$1$-dimensional:  $\langle h_{a_1, a_2, a_3}\rangle$; 
$\langle h_{a_1, a_2, a_3}\rangle \sim \langle h_{b_1, b_2, b_3}\rangle$  iff $(a_1, a_2, a_3)$ is a scalar multiple of  $(b_1, b_2,  b_3)$, or
 $(b_1, b_3,  b_2)$.  \\
&& $2$-dimensional: Spaces
orthogonal to $h_{a_1, a_2, a_3}$, with the same conjugacy condition. \\
\hline

28& $1$& $\mathfrak{t}$ can be any subspace of $\mathfrak{h}$, and all of them yield non-conjugate
regular subalgebras. \\
\hline

29& $s_2$& 
$1$-dimensional: $\langle h_{a_1, a_2, a_3}\rangle$; 
$\langle h_{a_1, a_2, a_3}\rangle \sim \langle h_{b_1, b_2, b_3}\rangle$  iff $(a_1, a_2, a_3)$ is a scalar multiple of  $(b_1, b_2,  b_3)$, or
 $(b_1, b_3,  b_2)$.  \\
&& $2$-dimensional: Spaces
orthogonal to $h_{a_1, a_2, a_3}$, with the same conjugacy condition. \\
\hline

30& $ s_3$  &$1$-dimensional: $\langle h_{a_1, a_2, a_3}\rangle$; 
$\langle h_{a_1, a_2, a_3}\rangle \sim \langle h_{b_1, b_2, b_3}\rangle$  iff $(a_1, a_2, a_3)$ is a scalar multiple of  $(b_1, b_2, \pm b_3)$.  \\
&& $2$-dimensional: Spaces
orthogonal to $h_{a_1, a_2, a_3}$, with the same conjugacy condition. \\
\hline

31& $s_1s_3$ &   $1$-dimensional: $\langle h_{a_1, a_2, a_3}\rangle$; 
$\langle h_{a_1, a_2, a_3}\rangle \sim \langle h_{b_1, b_2, b_3}\rangle$  iff $(a_1, a_2, a_3)$ is a scalar multiple   
of $(b_1, b_2, b_3)$, or $(b_2, b_1, -b_3)$.\\
&& $2$-dimensional: Spaces
orthogonal to $h_{a_1, a_2, a_3}$, with the same conjugacy condition. \\
\hline

32& $s_1, s_2$&  $1$-dimensional: $\langle h_{a_1, a_2, a_3}\rangle$; 
$\langle h_{a_1, a_2, a_3}\rangle \sim \langle h_{b_1, b_2, b_3}\rangle$  iff  $(a_1, a_2, a_3)$ is a scalar multiple   
of any permutation of the elements of $(b_1, b_2, b_3)$. \\
&& $2$-dimensional: Spaces
orthogonal to $h_{a_1, a_2, a_3}$, with the same conjugacy condition. \\
\hline

33& $1$& $\mathfrak{t}$ can be any subspace of $\mathfrak{h}$, and all of them yield non-conjugate
regular subalgebras.
\\
\hline

 \end{tabular}
\vspace{2mm}
\caption{Special closed subsets  $T$--row numbers $23$ to $33$--of $C_3$;  the generators of the stabilizer of $T$,  denoted $W_T$; and the $1$- and $2$-dimensional subalgebras $\mathfrak{t}$ of the Cartan subalgebra of the regular subalgebras determined by $T$. In this table $h_{a_1, a_2, a_3}=\mathrm{diag}(a_1, a_2, a_3, -a_1, -a_2, -a_3)$.}
\end{table}
\end{landscape}
\end{footnotesize}

\begin{footnotesize}
\begin{landscape}
\renewcommand{\arraystretch}{1.5}
\begin{table}[h!]
\begin{tabular}{|l|c|l|}
  \hline 
   \rowcolor{Gray}
  $\mathbf{T}$ & $\mathbf{W_T}$  & $\mathbf{\mathfrak{t}}$ \\ 
  \hline
34& $1$& $\mathfrak{t}$ can be any subspace of $\mathfrak{h}$, and all of them yield non-conjugate
regular subalgebras.\\ \hline

35& $s_3, s_2s_3s_2$&  $1$-dimensional: $\langle h_{a_1, a_2, a_3}\rangle$; 
$\langle h_{a_1, a_2, a_3}\rangle \sim \langle h_{b_1, b_2, b_3}\rangle$  iff $(a_1, a_2, a_3)$ is a scalar multiple   
of $(b_1, \pm b_2, \pm b_3)$\\
&& (all sign combinations included).\\
&& $2$-dimensional: Spaces
orthogonal to $h_{a_1, a_2, a_3}$, with the same conjugacy condition. \\
\hline

36& $s_2s_3s_2, s_2s_1s_2$& 
$1$-dimensional: $\langle h_{a_1, a_2, a_3}\rangle$; 
$\langle h_{a_1, a_2, a_3}\rangle \sim \langle h_{b_1, b_2, b_3}\rangle$  iff $(a_1, a_2, a_3)$ is a scalar multiple   
of $(b_1, \pm b_2,  b_3)$, or $(b_3, \pm b_2,  b_1)$. \\
&& $2$-dimensional: Spaces
orthogonal to $h_{a_1, a_2, a_3}$, with the same conjugacy condition. \\
\hline

37& $s_1,  s_2$& 
$1$-dimensional: $\langle h_{a_1, a_2, a_3}\rangle$; 
$\langle h_{a_1, a_2, a_3}\rangle \sim \langle h_{b_1, b_2, b_3}\rangle$  iff $(a_1, a_2, a_3)$ is a scalar multiple    \\
&& of any permutation of the elements of $(b_1, b_2, b_3)$.\\
&& $2$-dimensional: Spaces
orthogonal to $h_{a_1, a_2, a_3}$, with the same conjugacy condition. \\
\hline

38& $s_3s_2s_3$& 
$1$-dimensional: $\langle h_{a_1, a_2, a_3}\rangle$; 
$\langle h_{a_1, a_2, a_3}\rangle \sim \langle h_{b_1, b_2, b_3}\rangle$  iff  $(a_1, a_2, a_3)$ is a scalar multiple   
of  $(b_1, b_2, b_3)$, or $(b_1, -b_3, -b_2)$. \\
&& $2$-dimensional: Spaces
orthogonal to $h_{a_1, a_2, a_3}$, with the same conjugacy condition. \\
\hline

39& $s_1s_2s_1$&  $1$-dimensional: $\langle h_{a_1, a_2, a_3}\rangle$; 
$\langle h_{a_1, a_2, a_3}\rangle \sim \langle h_{b_1, b_2, b_3}\rangle$  iff $(a_1, a_2, a_3)$ is a scalar multiple of  $(b_1, b_2, b_3)$, or $(b_3, b_2, b_1)$.  \\
&& $2$-dimensional: Spaces
orthogonal to $h_{a_1, a_2, a_3}$, with the same conjugacy condition. \\
\hline

40& $ s_1, s_3$&  
$1$-dimensional: $\langle h_{a_1, a_2, a_3}\rangle$; 
$\langle h_{a_1, a_2, a_3}\rangle \sim \langle h_{b_1, b_2, b_3}\rangle$  iff  $(a_1, a_2, a_3)$ is a scalar multiple of  $(b_1, b_2,  \pm b_3)$, or
 $(b_2, b_1,  \pm b_3)$.  \\
 && $2$-dimensional: Spaces
orthogonal to $h_{a_1, a_2, a_3}$, with the same conjugacy condition. \\
\hline

41& $s_1s_2s_1$&   $1$-dimensional: $\langle h_{a_1, a_2, a_3}\rangle$; 
$\langle h_{a_1, a_2, a_3}\rangle \sim \langle h_{b_1, b_2, b_3}\rangle$  iff  $(a_1, a_2, a_3)$ is a scalar multiple of  $(b_1, b_2, b_3)$, or $(b_3, b_2, b_1)$.  \\
 && $2$-dimensional: Spaces
orthogonal to $h_{a_1, a_2, a_3}$, with the same conjugacy condition. \\
\hline

42&  $ s_2s_3s_2$&   $1$-dimensional: $\langle h_{a_1, a_2, a_3}\rangle$; 
$\langle h_{a_1, a_2, a_3}\rangle \sim \langle h_{b_1, b_2, b_3}\rangle$  iff $(a_1, a_2, a_3)$ is a scalar multiple of  $(b_1, \pm b_2,  b_3)$.  \\
&& $2$-dimensional: Spaces
orthogonal to $h_{a_1, a_2, a_3}$, with the same conjugacy condition. \\
\hline

43& $s_2s_3s_2, s_1s_2s_1$& 
 $1$-dimensional: $\langle h_{a_1, a_2, a_3}\rangle$; 
$\langle h_{a_1, a_2, a_3}\rangle \sim \langle h_{b_1, b_2, b_3}\rangle$  iff $(a_1, a_2, a_3)$ is a scalar multiple of $(b_1, \pm b_2, b_3)$, or $(b_3, \pm b_2, b_1)$.  \\
&& $2$-dimensional: Spaces
orthogonal to $h_{a_1, a_2, a_3}$, with the same conjugacy condition. \\
\hline

44& $s_1s_2s_1, s_1s_3s_2s_1$& 
$1$-dimensional: $\langle h_{a_1, a_2, a_3}\rangle$; 
$\langle h_{a_1, a_2, a_3}\rangle \sim \langle h_{b_1, b_2, b_3}\rangle$  iff $(a_1, a_2, a_3)$ is a scalar multiple of $(b_1,  b_2, b_3)$, or $(\pm b_3,  b_2, b_1)$.  \\
&& $2$-dimensional: Spaces
orthogonal to $h_{a_1, a_2, a_3}$, with the same conjugacy condition. \\
\hline

 \end{tabular}
\vspace{2mm}
\caption{Special closed subsets  $T$--row numbers $34$ to $44$--of $C_3$;  the generators of the stabilizer of $T$,  denoted $W_T$; and the $1$- and $2$-dimensional subalgebras $\mathfrak{t}$ of the Cartan subalgebra of the regular subalgebras determined by $T$. In this table $h_{a_1, a_2, a_3}=\mathrm{diag}(a_1, a_2, a_3, -a_1, -a_2, -a_3)$.}
\end{table}
\end{landscape}
\end{footnotesize}

\begin{scriptsize}
\begin{landscape}
\renewcommand{\arraystretch}{1.5}
\begin{table}[h!]
\begin{tabular}{|l|c|l|}
  \hline 
   \rowcolor{Gray}
  $\mathbf{T}$ & $\mathbf{W_T}$  & $\mathbf{\mathfrak{t}}$ \\ 
  \hline
45 & $s_3, s_2s_3s_2$&  $1$-dimensional: $\langle h_{\alpha_3}\rangle$. \\
 &&$2$-dimensional: Spaces  
   orthogonal to $h_{a_1, a_2, 0}$. The $2$-dimensional space determined by $(a_1, a_2, 0)$ is conjugate \\
   && to that determined by $(b_1, b_2, 0)$  iff
 $(a_1, a_2, 0)$ is a scalar multiple of  $(b_1, \pm b_2, 0)$.
\\ \hline

46& $s_1$&   $1$-dimensional: $\langle h_{\alpha_1}\rangle$.  \\
&& $2$-dimensional: Spaces 
   orthogonal to $\mathrm{diag}(a_1, a_1, a_3, -a_1, -a_1, -a_3)$, and all of them yield non-conjugate
regular subalgebras.\\
\hline

47 & $s_3$ &   $1$-dimensional: $\langle h_{\alpha_3}\rangle$.  \\
&& $2$-dimensional: Spaces 
   orthogonal to $\mathrm{diag}(a_1, a_2, 0, -a_1, -a_2, 0)$, and all of them yield non-conjugate
regular subalgebras.\\
\hline

48 & $s_2$& $1$-dimensional: $\langle h_{\alpha_2}\rangle$.  \\
&& $2$-dimensional: Spaces 
   orthogonal to $\mathrm{diag}(a_1, a_2, a_2, -a_1, -a_2, -a_2)$, and all of them yield non-conjugate
regular subalgebras.
\\
\hline

49 & $s_1, s_3$&   $1$-dimensional:  $\langle h_{\alpha_1}\rangle$.  \\
&& $2$-dimensional: Spaces orthogonal to 
  $h_{a_1, a_1, a_3}$.  The $2$-dimensional space determined by $(a_1, a_1, a_3)$ is conjugate to that \\ &&determined by $(b_1, b_1, b_3)$  iff
 $(a_1, a_1, a_3)$ is a scalar multiple of  $(b_1, b_1, \pm b_3)$.
\\
\hline

50 &
$s_1, s_3$& $1$-dimensional: $\langle h_{\alpha_3}\rangle$.  \\
&&$2$-dimensional: Spaces orthogonal to
  $h_{a_1, a_2,0}$.  The $2$-dimensional space determined by $(a_1, a_2, 0)$ is conjugate to that  \\ && determined by $(b_1, b_2, 0)$ iff
 $(a_1, a_2, 0)$ is a scalar multiple of  $(b_1, b_2, 0)$, or $(b_2, b_1, 0)$.\\
 \hline

51& $s_1, s_3$&  $\langle h_{\alpha_1}, h_{\alpha_3}\rangle$ \\
\hline

52 & $s_3$& $1$-dimensional: $\langle h_{\alpha_3}\rangle$.  \\
&&$2$-dimensional: Spaces  
   orthogonal to $h_{a_1, a_2, 0}$, and all of them yield non-conjugate
regular subalgebras.  \\
\hline

53& $s_1, s_2$&  $\langle h_{\alpha_1}, h_{\alpha_2}\rangle$ \\
\hline

54& $s_1$&  $1$-dimensional: $\langle h_{\alpha_1}\rangle$.  \\
&&$2$-dimensional: Spaces  
   orthogonal to $h_{a_1, a_1, a_3}$, and all of them yield non-conjugate
regular subalgebras.\\
\hline

55& $s_3$&  $1$-dimensional:  $\langle h_{\alpha_3}\rangle$. \\
 &&$2$-dimensional: Spaces  
   orthogonal to $h_{a_1, a_2, 0}$, and all of them yield non-conjugate
regular subalgebras. \\
\hline

56& $s_1s_2s_1$& $1$-dimensional: $\langle h_{\alpha_1+\alpha_2}\rangle$.  \\
&&
 $2$-dimensional: Spaces   
   orthogonal to $h_{a_1, a_2, a_1}$, and all of them yield non-conjugate regular subalgebras.
\\ \hline

 \end{tabular}
\vspace{2mm}
\caption{Levi decomposable closed subsets  $T$--row numbers $45$ to $56$--of $C_3$;  the generators of the stabilizer of $T$,  denoted $W_T$; and the $1$- and $2$-dimensional subalgebras $\mathfrak{t}$ of the Cartan subalgebra of the regular subalgebras determined by $T$. In this table $h_{a_1, a_2, a_3}=\mathrm{diag}(a_1, a_2, a_3, -a_1, -a_2, -a_3)$.}
\end{table}
\end{landscape}
\end{scriptsize}

\begin{footnotesize}
\begin{landscape}
\renewcommand{\arraystretch}{1.5}
\begin{table}[h!]
\begin{tabular}{|l|c|l|}
  \hline 
   \rowcolor{Gray}
  $\mathbf{T}$ & $\mathbf{W_T}$  & $\mathbf{\mathfrak{t}}$ \\ 
  \hline

57 & $s_2, s_3$&   $\langle h_{\alpha_2}, h_{\alpha_3}\rangle$ \\
\hline

58 & $s_3, s_2s_3s_2$& 
 $1$-dimensional: $\langle h_{\alpha_3}\rangle$. \\
 \hline
&& $2$-dimensional: Spaces are  
   orthogonal to $h_{a_1, a_2, 0}$.   The $2$-dimensional space determined by $(a_1, a_2, 0)$ is conjugate to that \\ && determined by $(b_1, b_2, 0)$  if and only if
 $(a_1, a_2, 0)$ is a scalar multiple of  $(b_1, \pm b_2, 0)$.\\
 \hline

59 &  $s_2, s_3s_2s_3$&   $1$-dimensional: $\langle h_{\alpha_2}\rangle$.   \\
&& $2$-dimensional: Spaces  
   orthogonal to $h_{a_1, a_2, a_2}$.   The $2$-dimensional space determined by $(a_1, a_2, a_2)$ is conjugate to that \\&& determined by $(b_1, b_2, b_2)$  if and only if
 $(a_1, a_2, a_2)$ is a scalar multiple of  $(b_1, \pm b_2, \pm b_3)$, or $(b_1, \pm b_3, \pm b_2)$ \\&&(all sign combinations included).
 \\
 \hline

60& $s_2, s_3$& $\langle h_{\alpha_2}, h_{\alpha_3}\rangle$\\
\hline

61& $s_2s_3s_2$& $1$-dimensional: $\langle h_{2\alpha_2+\alpha_3}\rangle$.   \\
&&
 $2$-dimensional: Spaces  
   orthogonal to $h_{a_1, 0, a_3}$, and all of them yield non-conjugate regular subalgebras.
\\ \hline

62 & $s_3$& $1$-dimensional: $\langle h_{\alpha_3}\rangle$.  
\\
&& $2$-dimensional: Spaces 
   orthogonal to $h_{a_1, a_2, 0}$, and all of them yield non-conjugate
regular subalgebras. 
\\\hline

63& $s_1s_2s_1$& $1$-dimensional: $\langle h_{\alpha_1+\alpha_2}\rangle$. \\
&& $2$-dimensional: Spaces 
   orthogonal to $h_{a_1, a_2, a_1}$, and all of them yield non-conjugate regular subalgebras.
\\
\hline

64 & $s_3$& $1$-dimensional:  $\langle h_{\alpha_3}\rangle$.  \\
&& $2$-dimensional:  Spaces 
   orthogonal to $h_{a_1, a_2, 0}$, and all of them yield non-conjugate
regular subalgebras. 
\\
\hline

65& $s_3, s_2s_3s_2$& $1$-dimensional: $\langle h_{2\alpha_2+\alpha_3}\rangle$.   \\
&& $2$-dimensional: Spaces 
   orthogonal to $h_{a_1, 0, a_3}$.
The $2$-dimensional space determined by $(a_1, 0, a_3)$ is conjugate to that \\ && determined by $(b_1, 0, b_3)$  if and only if
 $(a_1, 0, a_3)$ is a scalar multiple of  $(b_1, 0, \pm b_3)$.\\
 \hline

66& $s_3, s_2s_3s_2$& $\langle h_{\alpha_2}, h_{\alpha_3}\rangle$
\\
\hline

67& $s_2s_3s_2, s_2s_1s_2$& $1$-dimensional: $\langle h_{\alpha_1+\alpha_2}\rangle$. \\
&& $2$-dimensional: Spaces 
   orthogonal to $h_{a_1, a_2, a_1}$.
The $2$-dimensional space determined by $(a_1, a_2, a_1)$ is conjugate to that \\&& determined by $(b_1, b_2, b_1)$  if and only if
 $(a_1, a_2, a_1)$ is a scalar multiple of  $(b_1, \pm b_2, b_1)$.\\
 \hline

 \end{tabular}
\vspace{2mm}
\caption{Levi decomposable closed subsets  $T$--row numbers $57$ to $67$--of $C_3$;  the generators of the stabilizer of $T$,  denoted $W_T$; and the $1$- and $2$-dimensional subalgebras $\mathfrak{t}$ of the Cartan subalgebra of the regular subalgebras determined by $T$. In this table $h_{a_1, a_2, a_3}=\mathrm{diag}(a_1, a_2, a_3, -a_1, -a_2, -a_3)$.}
\end{table}
\end{landscape}
\end{footnotesize}

\begin{scriptsize}
\begin{landscape}
\renewcommand{\arraystretch}{1.5}
\begin{table}[h!]
\begin{tabular}{|l|c|l|}
  \hline 
   \rowcolor{Gray}
  $\mathbf{T}$ & $\mathbf{W_T}$  & $\mathbf{\mathfrak{t}}$ \\ 
  \hline

68 & $s_2s_3s_2, s_2s_1s_2$& $1$-dimensional: $\langle h_{2\alpha_2+\alpha_3}\rangle$. \\ 
 &&  $2$-dimensional: Spaces  
   orthogonal to $h_{a_1, 0, a_3}$.
The $2$-dimensional space determined by $(a_1, 0, a_3)$ is conjugate to that \\&& determined by $(b_1, 0, b_3)$  if and only if
 $(a_1, 0, a_3)$ is a scalar multiple of  $(b_1, 0,  b_3)$, or $(b_3, 0,  b_1)$. \\
 \hline

69& $s_2s_3s_2, s_2s_1s_2$& $\langle h_{\alpha_1+\alpha_2}, 
 h_{2\alpha_2+\alpha_3}\rangle$. \\ \hline

70 &  $ s_3s_2s_3$& $1$-dimensional:  $\langle h_{\alpha_2+\alpha_3}\rangle$.  \\ 
&&  $2$-dimensional: Spaces  
   orthogonal to $h_{a_1, a_2, -a_2}$, and all of them yield non-conjugate regular subalgebras. \\ \hline

71& $s_1, s_3$& $1$-dimensional: $\langle h_{\alpha_3}\rangle$.  \\ &&  $2$-dimensional: Spaces orthogonal to
  $h_{a_1, a_2, 0}$.  The $2$-dimensional space determined by $(a_1, a_2, 0)$ is conjugate to that \\&& determined by $(b_1, b_2, 0)$  if and only if
 $(a_1, a_2, 0)$ is a scalar multiple of  $(b_1, b_2, 0)$, or $(b_2, b_1, 0)$. \\ \hline

72 & $s_1s_2s_1$& 
$1$-dimensional: $\langle h_{\alpha_1+\alpha_2}\rangle$.  \\
&& $2$-dimensional: Spaces  
   orthogonal to $h_{a_1, a_2, a_1}$, and all of them yield non-conjugate regular subalgebras. \\ \hline

73& $s_2s_3s_2$& 
$1$-dimensional: $\langle h_{2\alpha_2+\alpha_3}\rangle$.  \\ 
&&  $2$-dimensional: Spaces 
   orthogonal to $h_{a_1, 0, a_3}$,  and all of them yield non-conjugate regular subalgebras. \\ \hline

74 & $s_2s_3s_2, s_1s_2s_1$& $1$-dimensional : $\langle h_{2\alpha_2+\alpha_3}\rangle$.   \\
&& $2$-dimensional: Spaces  
   orthogonal to $\mathrm{diag}(a_1, 0, a_3, -a_1, 0, -a_3)$.
    The $2$-dimensional space determined by $(a_1, 0, a_3)$ \\ && is conjugate to that  determined by $(b_1, 0, b_3)$  if and only if
 $(a_1, 0, a_3)$ is a scalar multiple of  $(b_1, 0, b_3)$, or $(b_3, 0, b_1)$.
 \\ \hline

75& $s_1s_2s_1, s_1s_3s_2s_1$& $\langle 
 h_{\alpha_3}, h_{\alpha_1+\alpha_2}\rangle$ \\ \hline

76& $ s_3, s_1s_2s_3s_2s_1$& $1$-dimensional: $\langle h_{\alpha_3}\rangle$.  \\
&& $2$-dimensional: Spaces orthogonal to
  $h_{a_1, a_2, 0}$.  
   The $2$-dimensional space determined by $(a_1, a_2, 0)$ is conjugate to that \\&& determined by $(b_1, b_2, 0)$  if and only if
 $(a_1, a_2, 0)$ is a scalar multiple of  
  $(\pm b_1, b_2, 0)$.
\\
\hline

77& $s_1s_2s_1, s_1s_3s_2s_1s_3$& $1$-dimensional: $\langle h_{\alpha_1+\alpha_2}\rangle$.  \\
&& $2$-dimensional: Spaces  
   orthogonal to $h_{a_1, a_2, a_1}$.
The $2$-dimensional space determined by $(a_1, a_2, a_1)$ is conjugate to that \\ && determined by $(b_1, b_2, b_1)$  if and only if
 $(a_1, a_2, a_1)$ is a scalar multiple of  $(b_1, b_2, b_1)$, or $(-b_1, b_2, -b_1)$.\\ \hline

78& $s_1s_2s_1, s_1s_3s_2s_1$& $\langle 
 h_{\alpha_3}, h_{2\alpha_1+2\alpha_2+\alpha_3}\rangle$ \\ \hline

 \end{tabular}
\vspace{2mm}
\caption{Levi decomposable closed subsets  $T$--row numbers $68$ to $78$--of $C_3$;  the generators of the stabilizer of $T$,  denoted $W_T$; and the $1$- and $2$-dimensional subalgebras $\mathfrak{t}$ of the Cartan subalgebra of the regular subalgebras determined by $T$. In this table $h_{a_1, a_2, a_3}=\mathrm{diag}(a_1, a_2, a_3, -a_1, -a_2, -a_3)$.}

\end{table}
\end{landscape}
\end{scriptsize}

\begin{footnotesize}
\begin{landscape}
\renewcommand{\arraystretch}{1.5}
\begin{table}[h!]
\begin{tabular}{|l|c|l|}
  \hline 
   \rowcolor{Gray}
  $\mathbf{T}$ & $\mathbf{W_T}$  & $\mathbf{\mathfrak{t}}$ \\ 
  \hline  \hline
    79 &$s_1, s_3, s_2s_1s_3s_2s_1s_3s_2$& $1$-dimensional: $\langle h_{\alpha_1}\rangle$ \\
     && $2$-dimensional: Spaces orthogonal to $h_{a, a, b}$. The space orthogonal to $h_{a, a, b}$ is conjugate to that \\
  && orthogonal to $h_{a', a', b'}$ iff $(a',a',b')$ is a scalar multiple of $(a,a,\pm b)$.\\
  \hline
  80 &$s_1, s_3, s_2s_3s_2$& $1$-dimensional: $\langle h_{\alpha_3}\rangle$ \\
     && $2$-dimensional: Spaces ortogonal to $h_{a,b,0}$. The space orthogonal to $h_{a,b,0}$ is conjugate to that\\
  && orthogonal to $h_{a',b',0}$ iff $(a',b')$ is a scalar multiple of $(a, \pm b)$ or $(b, \pm a)$.\\
  \hline
\end{tabular}
\vspace{2mm}
\caption{Symmetric closed subsets  $T$ of $C_3$;  the generators of the stabilizer of $T$,  denoted $W_T$; and the $1$- and $2$-dimensional subalgebras $\mathfrak{t}$ of the Cartan subalgebra of the regular subalgebras determined by $T$. We include only those closed systems of rank $1$, as the others cases are straightforward. In this table $h_{a_1, a_2, a_3}=\mathrm{diag}(a_1, a_2, a_3, -a_1, -a_2, -a_3)$.}
\label{csregend}
\end{table}
\end{landscape}
\end{footnotesize}

\end{document}